\documentclass{article}       
\usepackage[utf8]{inputenc} %

\usepackage{vpform}

\usepackage{moreverb}								%
\usepackage{textcomp}								%
\usepackage{graphicx}								%
\usepackage{svg}
\usepackage{url}
\usepackage{enumerate}

\usepackage{setspace}
\usepackage{tocloft}

\usepackage{subcaption}
\usepackage{caption}
\usepackage{bm}
\usepackage[normalem]{ulem}

\usepackage{hyperref}

\usepackage{lineno}
\usepackage{booktabs}

\usepackage[section]{algorithm}
\usepackage{algpseudocode}

\usepackage{newunicodechar}
\usepackage{amsmath}
\usepackage{amsthm}
\usepackage{thmtools,thm-restate}
\usepackage{amssymb,verbatim}
\usepackage{mathabx} %
\usepackage{mathrsfs}
\usepackage{xfrac}

\usepackage{cleveref}

\makeatletter
\def\lstAZ{A, B, C, D, E, F, G, H, I, J, K, L, M, N, O, P, Q, R, S, T, U, V, W, X, Y, Z}
\def\lstaz{a, b, c, d, e, f, g, h, i, j, k, l, m, n, o, p, q, r, s, t, u, v, w, x, y, z}
\def\lstAZBB{B, C, D, E, F, G, H, I, J, K, L, M, N, O, P, Q, R, T, U, V, W, X, Y, Z}

\newcommand{\MkScr}[1]{\expandafter\def\csname s#1\endcsname{\mathscr{#1}}}
\newcommand{\MkUp}[1]{\expandafter\def\csname u#1\endcsname{\mathrm{#1}}}
\newcommand{\MkFrak}[1]{\expandafter\def\csname f#1\endcsname{\mathfrak{#1}}}
\newcommand{\MkCal}[1]{\expandafter\def\csname c#1\endcsname{\mathcal{#1}}}
\newcommand{\MkBB}[1]{\expandafter\def\csname #1#1\endcsname{\mathbb{#1}}}

\@for\@tempa :=\lstAZ\do{%
	\expandafter\MkScr \@tempa   %
	\expandafter\MkFrak \@tempa   %
	\expandafter\MkUp \@tempa  %
	\expandafter\MkCal \@tempa   %
		  }    
\@for\@tempa :=\lstaz\do{%
	\expandafter\MkUp \@tempa    }    
	
\@for\@tempa :=\lstAZBB\do{%
	\expandafter\MkBB \@tempa      }
	    
\makeatother

\DeclareMathOperator*{\argmin}{argmin}
\renewcommand{\epsilon}{\varepsilon}

\renewcommand{\tilde}{\widetilde}
\numberwithin{equation}{section}

\newtheorem{theorem}{Theorem}[section]
\newtheorem{lemma}[theorem]{Lemma}
\newtheorem{proposition}[theorem]{Proposition}

\newtheorem*{theorem*}{Theorem}
\newtheorem*{lemma*}{Lemma}
\newtheorem*{proposition*}{Proposition}
\newtheorem*{corollary*}{Corollary}

\theoremstyle{definition}
\newtheorem{definition}[theorem]{Definition}
\newtheorem{remark}[theorem]{Remark}

\newtheorem{assumption}{Assumption}
\newtheorem{example}[theorem]{Example}

\crefname{proposition}{Proposition}{Propositions}
\crefname{lemma}{Lemma}{Lemmata}
\crefname{theorem}{Theorem}{Theorems}
\crefname{assumption}{Assumption}{Assumptions}
\crefname{section}{Section}{Sections}
\crefrangelabelformat{assumption}{#3#1#4--#5\crefstripprefix{#1}{#2}#6}
\crefname{algorithm}{Algorithm}{Algorithms}
\crefname{remark}{Remark}{Remarks}
\crefname{appendix}{Appendix}{Appendices}
\crefname{figure}{Figure}{Figures}

\newcommand{\R}{\mathbb R}

\def\DHLhksqrt#1#2{%
\setbox0=\hbox{$#1\sqrt{#2\,}$}\dimen0=\ht0
\advance\dimen0-0.2\ht0
\setbox2=\hbox{\vrule height\ht0 depth -\dimen0}%
{\box0\lower0.4pt\box2}}

\let\oldpropto\propto
\renewcommand{\propto}{\ensuremath{\; \oldpropto \;}}

\title{
    \bf Controlled stochastic processes for simulated annealing
}

\author{%
    \bf \small Vincent Molin\thanks{Contact: \url{molinv@chalmers.se}} 
    \hspace{10pt} Axel Ringh
    \hspace{10pt} Moritz Schauer
    \hspace{10pt} Akash Sharma \\
    \small Department of Mathematical Sciences \\ 
    \small Chalmers University of Technology and the University of Gothenburg
}
\date{\small April 11, 2025}

\renewcommand{\cW}{\mathcal{W}}
\renewcommand{\cA}{\mathcal{A}}
\renewcommand{\sP}{\mathscr{P}}
\renewcommand{\RR}{\mathbb{R}}
\renewcommand{\EE}{\mathbb{E}}

\newcommand{\cov}{\textnormal{Cov}}
\newcommand{\var}{\textnormal{Var}}
\newcommand{\covt}{\cov_{\mu_t}}
\newcommand{\vart}{\var_{\mu_t}}

\newcommand{\vdot}[2]{\left\langle #1, #2 \right\rangle}
\newcommand{\supp}{\textrm{supp}}
\newcommand{\difletter}{d} 
\newcommand{\difempty}{\mathop{}\!}   %
\newcommand{\dif}[1]{\mathop{}\!\difletter#1}
\newcommand{\obj}{U}    %
\newcommand{\dens}{\pi}   %
\newcommand{\gen}{\cA}  %
\newcommand{\ddim}{d}  %
\newcommand{\Rdim}{{\RR^\ddim}} %
\newcommand{\Psp}{\sP} %
\newcommand{\poinc}{C_P} %
\newcommand{\wass}{\cW} %
\newcommand{\leb}{\dif{x}}
\newcommand{\dx}{\dif{x}}
\newcommand{\dy}{\dif{y}}
\newcommand{\dt}{\dif{t}}
\newcommand{\modgen}{\gen^o}
\newcommand{\bpsvel}{Y}
\newcommand{\Tmap}[2]{{T_{#1\to#2}}}
\newcommand{\Ttth}{\Tmap{\mu_t}{\mu_{t+h}}}
\newcommand{\Id}{\textrm{Id}}

\begin{document}

\maketitle

\begin{abstract}
Simulated annealing solves global optimization problems by means of a random walk in a cooling energy landscape based on the objective function and a temperature parameter. However, if the temperature is decreased too quickly, this procedure often gets stuck in suboptimal local minima.
In this work, we consider the cooling landscape as a curve of probability measures. We prove the existence of a minimal norm velocity field which solves the continuity equation, a differential equation that governs the evolution of the aforementioned curve. The solution is the weak gradient of an integrable function, which is in line with the interpretation of the velocity field as a derivative of optimal transport maps.
We show that controlling stochastic annealing processes by superimposing this velocity field would allow them to follow arbitrarily fast cooling schedules. Here we consider annealing processes based on diffusions and piecewise deterministic Markov processes. 
Based on convergent optimal transport-based approximations to this control, we design a novel interacting particle--based optimization method that accelerates annealing. We validate this accelerating behaviour in numerical experiments. 
\end{abstract}

\renewcommand{\cfttoctitlefont}{\large\bfseries}
\setlength\cftbeforesecskip{4pt}
\setcounter{tocdepth}{1}

\tableofcontents

\section{Introduction}

Stochastic exploration methods solve optimization problems by changing the argument of the optimization objective $\obj(x)$ randomly, for example by making a random walk in the energy landscape induced by the objective function. The magnitude of these random fluctuations can be cast as a temperature, which corresponds to different scalings of the energy landscape.
Simulated annealing \cite{kirkpatrick1983optimization, van1987simulated, dekkers1991global} is an important example of this. In its continuous-time diffusion-based version, introduced in \cite{geman1986diffusions, chiang1987diffusion}, the temperature $\beta(t)^{-1}$ is decreased over time to trap the random walker in an optimum. 
At a time $t$, the random walkers momentarily follow diffusive dynamics which, if kept unchanged, would mix towards the stationary measure $\mu_t$ with density function $ \exp(-\beta(t)\obj(x))$ up to proportionality, under suitable conditions on $\obj$.
If the temperature is decreased slowly enough, the random walker finds and becomes trapped in a \emph{global} optimum as time goes to infinity \cite{chiang1987diffusion}. 
If the temperature is decreased too quickly, the random walkers are not able to keep up with the changes in the energy landscape and get frozen in local minima.

In this work, we prove the existence of a vector field $v_t(x)$ which captures the changes in the cooling landscape, in the sense that it describes how particles in this landscape, for example, random walkers, have to be pushed to get arbitrarily close to a global optimum in bounded time. This vector field achieves this with minimal effort and arises as derivative \( v_t(x) = \lim_{h\to 0} h^{-1} \left(T_{\mu_t \to \mu_{t+h}}(x) - x\right)\) where $T_{ \mu_{t} \to  \mu_{t+h}}$ is the Monge optimal transport map (in $\|\cdot\|_2^2$-sense) transporting probability mass distributed as $ \mu_{t}$ to probability mass distributed as $ \mu_{t+h}$.
Then, as an application, we design and implement a novel interacting particle system--based optimization method that accelerates annealing using the vector field $v_t(x)$ approximated by discrete optimal transport.

So let $\obj\colon \Rdim \to \R$ be an objective function of interest and consider the optimization problem
\[
\min_{x\in \Rdim} \; \obj(x).
\]
Although many algorithms have been developed to solve these types of problems, see, e.g., one of the monographs \cite{bertsekas1999nonlinear, nocedal2006numerical}, it is still difficult to find globally optimal solutions if $\obj$ is non-convex.
One approach to such problems is named after the metallurgic technique called annealing, where a metal is heated and then slowly cooled to alter its physical properties by finding a more stable configuration.
The classical continuous-time annealing process for optimization in $\Rdim$, studied in, e.g., \cite{chiang1987diffusion, geman1986diffusions}, is given by the stochastic differential equation
\begin{equation}\label{eq:langevin}
    \dif X_t = -\nabla\obj(X_t) \dt + \sqrt{2\beta^{-1}(t)}\dif B_t,
\end{equation}
where $B_t$ is a $\ddim$-dimensional Brownian motion. 
Under fairly general conditions, this process is guaranteed to find a global minimum if the cooling schedule $\beta$ has at most logarithmic growth \cite{chiang1987diffusion, holley1989asymptotics}.%

In this work, we will consider the sequence of cooling measures as a curve in a suitable probability space. To this end,
for a fixed inverse temperature $\beta \in \RR_+$ we define the Gibbs measure corresponding to $\obj$ as
\[
    \mu_\beta(\dx) := Z_\beta^{-1}e^{-\beta\obj(x)}\dx,
\]
where $Z_\beta := \int e^{-\beta\obj(x)} \dx$ is the normalizing constant. For sufficiently regular functions $\obj$, the Gibbs measure $\mu_\beta$ concentrates on the set of global minima of $\obj$ as $\beta \to \infty$, in the sense that $\mu_\beta$ converges weakly to a limiting measure $\mu_\infty$ supported on the set of global minima of $\obj$ \cite{hwang1980laplace, hasenpflug2024wasserstein}.
Next, let $\beta(t)$ be a smooth function $\RR_+\to [1,\infty)$ and, by a slight abuse of notation, define the measure
\begin{equation}\label{eq:mu_t}
    \mu_t(\dx) := \mu_{\beta(t)}(\dx) = Z_{\beta(t)}^{-1}e^{-\beta(t)\obj(x)} \dx.
\end{equation}
We will refer to the map $t\mapsto\mu_t$ as the Gibbs curve.

For a given Markov process with time-dependent generator $\gen_t$ (in the sense of Stroock--Varadhan \cite[Chapter~6]{stroock1997multidimensional}), the Kolmogorov forward equation that governs the evolution of the law is given by 
\begin{equation}\label{eq:kolmogorov_forward}
    \partial_t\rho_t = \gen_t^*\rho_t.
\end{equation}
When $\gen_t$ is the generator of the diffusion process in \eqref{eq:langevin}, this corresponds to the PDE
\begin{equation}\label{eq:fokker_planck}
    \partial_t\rho_t = \nabla\cdot(\rho_t\nabla\obj) + \frac{1}{\beta(t)}\Delta\rho_t,
\end{equation}
where $\nabla\cdot$ denotes the divergence operator and $\partial_t\rho_t = \frac{\partial}{\partial t}\rho_t$. Further, and crucial to our point of view is now that, for this generator, it holds that
\[
    \cA^*_t\mu_t = 0 \ne \partial_t\mu_t, \quad \text{if $\beta'\ne 0$},
\]
and hence the Gibbs curve,
$t\mapsto \mu_t$, can not be a solution to Equation \eqref{eq:kolmogorov_forward}. 
Clearly, this holds for any time-inhomogeneous Markov process satisfying $\gen_t^* \mu_t = 0$, which is normally the type of dynamics constructed in the simulated annealing literature, see, for instance, \cite{geman1986diffusions, chiang1987diffusion,holley1988simulated, monmarche2016piecewise}.
A question, and the topic under investigation in this paper, is then, can we find a modified $\modgen_t$ such that  
\begin{equation}\label{eq:fp_question_intro}
    (\modgen_t)^* \mu_t = \partial_t\mu_t?
\end{equation}
This would then yield that the law of the process undergoes a deformation equivalent to that of the energy landscape, alleviating the need for logarithmic cooling over infinite time. Specifically, we are going to look for deterministic velocity fields $v_t$ such that
\begin{equation}\label{eq:ce_intro}
    \partial_t\mu_t + \nabla\cdot(v_t\mu_t) = 0,
\end{equation}
with the end goal of superimposing these dynamics on the simulated annealing process.
Viewing the Gibbs curve as a curve in an appropriate metric space of probability measures, one choice of $v_t$ arises as limits of optimal transport problems, see, e.g., \cite[Chapter~8]{ambrosio2008gradientflows}.
We first ensure that the theoretical underpinnings guaranteeing the existence of this velocity are satisfied.
We then take inspiration from this and design an interacting particle system algorithm, where we incorporate solving discrete optimal transport problems along the way.

\subsection{Related work}\label{sec:related_work}

It is possible to understand the simulated annealing process \eqref{eq:langevin} in a number of other ways. For example, the Jordan--Kinderlehrer--Otto formalism identifies the Langevin dynamics in \eqref{eq:langevin} as the Wasserstein-2 gradient flow of the functional 
\begin{equation}\label{eq:jko}
    F_\beta(\mu) = \int U \dif \mu + \frac{1}{\beta} \int \log\left(\frac{\dif \mu}{\dif \lambda}\right) \dif \mu,
\end{equation}
with $\lambda$ the Lebesgue measure on $\Rdim$, see, e.g., \cite{jordan1998variational,ambrosio2008gradientflows}. Hence, we can view the annealing process as an inhomogeneous gradient flow of the functionals $F_{\beta(t)}$. In general, this gradient flow
converges at a speed that can be exponentially slow in $\beta$, see Section \ref{sec:finite_end}, which is related to the fact that $\beta(t)$ can grow at most logarithmically fast in order to guarantee convergence \cite{chiang1987diffusion, monmarche2016piecewise}.

Different perspectives on simulated annealing has lead to different ways of constructing new optimization methods. For instance, from the gradient flow formulation described above, \cite{bolte2024swarm} consider a new family of methods based on modifying the functional in \eqref{eq:jko}. Orthogonally, \cite{monmarche2016piecewise} instead consider replacing the diffusive dynamics in \eqref{eq:langevin} with piecewise deterministic Markov processes.

From a more algorithmic perspective, the method we design bears a resemblance to other particle-based optimization approaches, such as consensus-based optimization \cite{carrillo2018analytical, tretyakov2023consensus} and sequential Monte Carlo simulated annealing \cite{zhou2013sequential}. More precisely, consensus-based optimization orchestrates interaction among particles where each particle drifts towards a weighted average, where higher weights are assigned to particles if they currently are, relative to the other particles, in positions corresponding to better objective values. In contrast, we assign to each particle an individual target by first solving a population-wide discrete optimal transport problem. This introduces an inhomogeneous between the particles. For ease of presentation, we defer further elaboration on this to \cref{rem:cbo}.
Finally, while we will look for a velocity field that follows the Gibbs curve by continuously moving the particles, it is possible to understand sequential Monte Carlo simulated annealing \cite{zhou2013sequential} as achieving this curve-following by successively removing particles from low probability regions and resampling them in high probability regions. We elaborate on this in \cref{rem:smcsa}. 

In the context of Bayesian inference, a similar velocity field formulation can be found in, e.g, \cite{heng2021gibbs}. While we approximate the velocity field using a particle system and discrete optimal transport, in \cite{heng2021gibbs} the authors instead consider computing an inexact but numerically more tractable velocity field. Approximate optimal transport maps have also been investigated for use in Bayesian inference, see for instance \cite{parno2018transport}.

    As a final observation here, we note that our view on simulated annealing is reminiscent of
    interior points methods in optimization, following a barrier trajectory \cite{forsgren2002interior},  \cite[Chapter~14]{nocedal2006numerical}, \cite[Chapter~11]{boyd2004convex}, \cite[Part 3]{vanderbei2008linear}. To this end, note that we can identify the finite-dimensional, potentially non-convex, optimization problem as an infinite-dimensional linear optimization problem over the set of probability measures.
    For a continuously differentiable function $U$ with a compact set of global minimizers (see Assumption~\ref{ass:inf_obj_exists}),
    we have that
    \[
        \min_{x \in \Rdim} U(x) = \min_{\rho \in \Psp} \int U(x) \difempty\rho(\dx),
    \]
    where $\Psp$ denotes the set of probability measures on $\Rdim$. The minimum at the right-hand side is attained in, e.g., $\delta_{x^*}(\dx)$, where $\delta_x(\dx)$ denotes the Dirac measure at $x$ and where $x^*$ is any
    global minimizer of $U$.
    Now, let $\Pi$ be the set of measures with strictly positive Lebesgue probability densities on $\RR^d$, that is
    \[
    \Pi := \{ \rho \in \Psp(\Rdim) \colon \rho(\dx) = \pi(x)\dx, \; \pi(x) > 0 \, \forall x \in \Rdim\}.
    \]
    Note that any minimizer of the form $\delta_{x^*}(\dx)$ is in the (weak) closure of $\Pi$, and therefore, with a slight abuse of notation, we have that
    \[
    \min_{\rho \in \Psp} \int U(x) \difempty\rho(\dx) = \inf_{\rho \in \Pi} \int U(x) \difempty\rho(\dx) = \inf_{\pi \in \Pi} \int U(x) \pi(x) \dx
    \]
    Indeed, by introducing a regularization term parametrized by $\beta$ we can consider the problem
    \[
        \min_{\pi \in \Pi} \int U(x)\pi(x) \dx  + \frac{1}{\beta} \int \pi(x) \log\pi(x) \dx,
    \]
    which has the solution $ \pi_\beta \propto \exp(-\beta U) \in \Pi$. While $\Pi$ is not the interior of $\Psp$, following the Gibbs curve $t \mapsto \pi_t$ can still be seen as an analogue to following the barrier trajectory.

\subsection{Outline}

The outline of the paper is as follows: in \cref{sec:background} we introduce notation and give a brief outline of relevant results from the literature, in particular the metric side of the spaces of probability measures and curves in such spaces. In \cref{sec:velocity_field}, we show that the Gibbs curve $t \mapsto \mu_t$ is absolutely continuous, by establishing the existence of a velocity field $v_t$ which solves \eqref{eq:ce_intro}. Using this velocity field, we then proceed in \cref{sec:controlled_superposition} to show that there exists both a diffusion process and a piecewise deterministic Markov process that at any time $t$ have a law agreeing with $\mu_t$, hence giving a definite answer to \eqref{eq:fp_question_intro}. In \cref{sec:num_approx}, we leverage these results to design interacting particle--based optimization algorithms. To do so,
we show a fairly general convergence result for transport map estimation relying on self-normalized reweighting, and then estimate the velocity field by solving discrete optimal transport problems. Finally, in \cref{sec:numerical_experiments} we illustrate and evaluate our methods on a number of test problems.
For improved readability, the paper also contains \cref{app:proofs} to which we defer some of our proofs.

\section{Background} \label{sec:background}

In this section we introduce background material needed for the rest of the paper. For ease of reference, we first set up the notation used throughout. 

\subsection{Notation}
For density functions $f,g$, \( f \propto g \) means that there exists a constant $a > 0$ such that \( f = a g. \)
We use $\leb$ as shorthand notation for the standard Lebesgue measure on $\Rdim$, and read $\mu(\dx) = \pi(x) \dx$ as the measure $\mu$ with Radon--Nikodym derivative, or density, $\pi$ with respect to the Lebesgue measure $\leb$. For measures $\mu,\nu$ then $\mu \ll \nu$ if and only if $\mu$ is dominated by $\nu$. We denote by $|\cdot|$ the Euclidean norm, by $\langle\cdot,\cdot\rangle$ the Euclidean inner product, and by $\|\cdot\|_{\text{Fr}}$ the Frobenius norm of a matrix.

In the following, we need a number of different spaces. By $C(E,Y)$ we denote the set of $Y$-valued continuous functions on $E$. When $Y = \RR$ we write $C(E)$, and by $C^k(E)$ we denote the set of $\RR$-valued $k$-times continuously differentiable functions on $E$. $C^k_c(E)$ is the set of compactly supported $f \in C^k(E)$.
By $\Psp(\Rdim)$ we denote the probability measures on $\Rdim$, and for $\mu\in \Psp(\Rdim)$,
we write with a slight abuse of notation $L^p(\mu) = L^p(\mu; \Rdim, \RR^\ell)$ for
the set of functions $f \colon \Rdim \to \RR^\ell$ such that $\| f \|_{L^p(\mu)} := \int_{\Rdim} | f |^p \dif \mu < \infty$, with $\ell = 1$ or $\ddim$ depending on the context.
$\Psp_p(E)$ is the set of probability measures on the metric space $(E,d)$ with bounded $p$-th moments, i.e., $\mu \in \Psp_p(E)$ if and only if there exists a point $x_0 \in E$ such that $\int d(x,x_0)^p \mu(\dx) < \infty$. For a measure $\mu$ or a function $f$, we denote by $\supp(\mu)$, $\supp(f)$ the support of $\mu$ and $f$, respectively. 

We define the ($L^2$-)Poincar\'e constant $\poinc$ of $\mu \in \sP(\Rdim)$ as
\[
    \poinc(\mu) := \inf \,\, \left\{C \geq 0 \colon  \|f\|_{L^2(\mu)} \leq C  \| \nabla f\|_{L^2(\mu)}, \quad \forall f\in C^\infty_c(\Rdim)  \right\}.
\]
A Borel probability measure $\mu \in \sP(\Rdim)$ with a finite Poincar\'e constant, $C_P(\mu) < \infty$, defines a norm on the set of smooth compactly supported functions $C^\infty_c(\Rdim)$ by the pairing
\[
    \langle f ,  g\rangle_\mu  = \int fg \dif \mu + \int \langle \nabla f, \nabla g\rangle \dif \mu.
\]
Denote by \( H^1(\mu) \) the completion of $C^\infty_c(\Rdim)$ with respect to the norm induced by $\langle\cdot,\cdot\rangle_\mu$. By the Meyers-Serrin theorem (see \cite[Chapter~7]{lieb2001analysis}), $H^1(\mu)$ is the space of $L^2$-functions with weak derivatives in $L^2(\mu)$, that is,
\begin{align*}
    H^1(\mu) &= \{ f \in L^2(\mu; \Rdim, \RR) \colon \nabla f \in L^2(\mu; \Rdim, \Rdim ) \}.%
\end{align*}

We say that a sequence of probability measures $(\mu_n)$ converges weakly to a probability measure $\mu$ iff 
\[
    \lim_{n\to\infty} \int f \dif\mu_n \to \int f \dif\mu
\]
for every continuous bounded function $f$. For a probability measure $\mu$ and $f,g\in L^2(\mu)$ we write $\cov_\mu(f,g) := \int (f-\int f \dif\mu)(g-\int g \dif \mu)\dif\mu$ and $\var_\mu(f) := \cov_\mu(f,f)$.

\subsection{Optimal transport and curves in the metric space \texorpdfstring{$(\Psp_p(\Rdim),\wass_p)$}{(P(Rd), Wp)}}

Here we state definitions and results related to the metric side of spaces of probability measures, which are central to our arguments. We follow the comprehensive treatment of \cite{ambrosio2008gradientflows}.
The space of probability measures $\Psp_p(\Rdim)$ for $p \geq 1$ %
becomes a metric space when equipped with the $p$-th Wasserstein distance $\wass_p$, 
defined via the optimal transport problem
\begin{equation}\label{eq:wass_p}
    \wass_p^p(\mu,\nu) := \inf_{%
        \substack{%
            \gamma\in\Psp(\Rdim\times \Rdim) \\
            \int\gamma(\cdot,\dx) = \mu \\
            \int\gamma(\dx,\cdot) = \nu
        }
    }
    \EE_{(X,Y)\sim\gamma}\left[ |X-Y|^p \right]
    .
\end{equation}
The set $\Gamma(\mu,\nu) := \{ \gamma\in\Psp(\Rdim\times \Rdim) \colon \int\gamma(\cdot,\dx) = \mu, \int\gamma(\dx,\cdot) = \nu\}$ is the admissible couplings, or transport plans, between $\mu$ and $\nu$. Intuitively, a joint distribution $\gamma^* \in \Gamma(\mu,\nu)$ achieving the infimum in \eqref{eq:wass_p} describes a minimal cost deformation of the probability mass distributed according to $\mu$ into $\nu$.
Endowing $\Psp_p(\Rdim)$ with $\wass_p$ yields a (complete and separable) metric space \cite[Theorem~6.18]{villani2008optimal}.

If the measure we transport from has a Lebesgue density, then the optimal transport plan $\gamma^*$ takes the form of a deterministic transport map $T$. 
\begin{proposition}[Existence of Monge transport map, {\cite[Theorems~2.44~and~2.50]{villani2003topics}}]\label{prop:existence_monge_map}
    Let $\mu,\nu\in\Psp_p(\Rdim)$. If $\mu$ has a Lebesgue density then there exists a Monge map solving \eqref{eq:wass_p}, that is, a measurable map $T\colon\Rdim\to\Rdim$ such that $\nu=T_\#\mu$ and
    \[
        \wass_p (\mu,\nu) = \left( \int |x-T(x)|^p \difempty\mu(\dx) \right)^{\frac{1}{p}}.
    \]
    This corresponds to the deterministic optimal coupling $\gamma^*$, given by
    \[
        \gamma^* = (\textnormal{Id}, T)_\#\mu.
    \]
    Further, when $p>1$, $T$ is $\mu$-almost everywhere unique and we denote by $\Tmap{\mu}{\nu}$ the optimal transport map from $\mu$ to $\nu$.
\end{proposition}

\begin{definition}[Absolutely continuous curves and metric derivatives, {\cite[Section~1.1]{ambrosio2008gradientflows}}]
    We call a curve\footnote{A curve is a continuous map of an interval $I \subset \RR$ into a topological space.} $\eta\colon I\to\Psp_p(\Rdim),\, t\mapsto \mu_t$ \emph{absolutely continuous} if the \emph{metric derivative} $|\mu'|(t)$, defined as the limit
    \[
        |\mu'|(t) := \lim_{h\to 0} \frac{\wass_p(\mu_t,\mu_{t+h})}{h},
    \]
    exists for Lebesgue almost every $t$ and belongs to $L^1(I)$.
\end{definition}

We can now make the continuity equation \eqref{eq:ce_intro} formal.

\begin{theorem}[Absolutely continuous curves and velocity fields, {\cite[Theorem~8.3.1]{ambrosio2008gradientflows}}]\label{thm:ac_curves_and_velocity_fields}
    Let $I \mapsto \mu_t$ be an absolutely continuous curve in $(\sP_p(\Rdim), \wass_p)$. Then there exists a time-dependent velocity field $v_t(x)\colon I \times \Rdim \to \Rdim$ satisfying the continuity equation
    \begin{equation}\label{eq:focker_planck_v}
        \partial_t \mu_t + \nabla\cdot(v_t\mu_t) = 0,
    \end{equation}
    in the sense that for all compactly supported smooth functions $f \in C^\infty_c(I \times \Rdim)$
        \begin{equation}\label{eq:weak_vel_form}
            \int_I\int_\Rdim\left(\partial_t f + \langle v_t, \nabla_x f\rangle\right) \dif\mu_t\dt = 0.
        \end{equation}
    Further, it holds that $\|v_t\|_{L^p(\mu_t,\Rdim)} \in L^1(I)$.

    Conversely, assume that there exists a $v_t$ that solves \eqref{eq:focker_planck_v} and such that $\|v_t\|_{L^p(\mu_t,\Rdim)} \in L^1(I)$. Then $t\mapsto\mu_t$ is an absolutely continuous curve.
\end{theorem}

\begin{remark}
    Absolute continuity for a curve defined on an interval $[0,T]$ coincides with a finite curve length. More precisely, the length $\ell$ of a curve $t\mapsto\mu_t$ is defined by
    \begin{align}
        \ell(t\mapsto\mu_t) := \int_0^T |\mu'|(t) \dt=\, & \inf_v  \quad\int_0^T \|v_t\|_{L^2(\mu_t)} \dt, \label{eq:flow_formulation}                                     \\
                                 & \text{s.t.} \quad \partial_t\mu_t + \nabla\cdot(v_t\mu_t) = 0, \label{eq:continuity_equation}
    \end{align}
    using the convention that the infimum is infinite if there are no velocity fields $v$ satisfying the continuity equation. 
\end{remark}

For absolutely continuous curves of measures $t \mapsto \mu_t$ with Lebesgue densities, the existence of Monge transport maps in \cref{prop:existence_monge_map} gives another characterization of the velocity fields in Theorem \ref{thm:ac_curves_and_velocity_fields}.
\begin{proposition}[Transport map characterization of $v_t$, {\cite[Proposition~8.4.6]{ambrosio2008gradientflows}}]\label{prop:v_as_monge_limit}
    Let $t\mapsto \mu_t\in(\Psp_p(\Rdim),\wass_p)$ be an absolutely continuous curve, with $1<p<\infty$. If $\mu_t$ has a Lebesgue density for all $t$, then
    \[
        v_t = \lim_{h\to 0} h^{-1}(\Tmap{\mu_t}{\mu_{t+h}} - \textnormal{Id}) \quad \text{ in } L^2(\mu_t),
    \]
    where $\Tmap{\mu_t}{\mu_{t+h}}$ is the unique Monge transport map from $\mu_t$ to $\mu_{t+h}$.
\end{proposition}

\subsection{Simulated annealing on \texorpdfstring{$\Rdim$}{Rd}}

The stochastic differential equation
\[
    \dif X_t = -\nabla\obj(X_t) \dt + \sqrt{2\beta^{-1}}\dif B_t,
\]
 is referred to as the (overdamped) Langevin dynamics. Here, $B_t$ is a Brownian motion on $\Rdim$. This stochastic process has generator
\begin{equation}\label{eq:generator_langevin}
    \cL_\beta f = \frac{1}{\beta}\Delta f - \langle \nabla f, \nabla\obj\rangle.
\end{equation}
Under suitable conditions on $\obj$, this process has the Gibbs measure $\mu_\beta = Z_\beta^{-1}e^{-\beta\obj(x)}\dx$ as its unique invariant measure. 
Simulated annealing on $\Rdim$ \cite{van1987simulated,dekkers1991global,chiang1987diffusion} is the stochastic process with time-dependent diffusion coefficient given by the Langevin-like SDE
\begin{equation}\label{eq:simulated_annealing_sde}
    dX_t = -\nabla\obj(X_t)\dt + \sqrt{2\beta(t)^{-1}} \dif B_t.
\end{equation}
For each potential $\obj$ belonging to a fairly large class, there exists a constant $C_\obj > 0$ such that if \( \beta(t) \leq C_\obj \log(1+t)\), then, for any initial distribution, the solution of \eqref{eq:simulated_annealing_sde} converges weakly to $\mu_\infty$ as $t\to\infty$ \cite{chiang1987diffusion}. 

Conversely, one can consider the rate at which the Gibbs curve itself approaches its limiting measure $\mu_\infty$. %
Assuming that $\obj$ permits a unique global minimum $x^*$ with a positive definite Hessian $\nabla^2 \obj(x^*)$, then \cite[Theorem~3.8]{hasenpflug2024wasserstein} gives the following concentration rate for the Gibbs measure
\begin{equation}
    \wass_p(\mu_t, \mu_\infty) \leq K \beta(t)^{-1/2},
\end{equation}
for some positive constant $K$.

\section{A velocity field solution to the continuity equation} \label{sec:velocity_field}

In this section, we prove that the Gibbs curve, as a map from a finite interval $[0,T]$ to $\Psp_2(\Rdim)$, is absolutely continuous. For this, we make the following assumptions on $\obj$.

\begin{assumption}\label{ass:inf_obj_exists}
    $U$ is of class $C^1$ and satisfies $\inf_{x} \obj(x) > -\infty$. Moreover, $\{ x \in \Rdim \mid x = \inf_{x} \obj(x) \}$ is non-empty and compact. Without loss of generality, we also assume that $\inf_{x} \obj(x) = \min_{x} \obj(x) = 0$.
\end{assumption}

\begin{assumption}\label{ass:quadratic_tail}
    There exist $\alpha >0,$ $R > 0$ such that $|x| > R \implies \obj(x) \geq \alpha |x|^2$.
\end{assumption}

\begin{assumption}\label{ass:quadratic_gradient}
    There exist $\alpha > 0, R > 0$ such that \( |x| > R \implies    \langle\nabla\obj(x),x\rangle \geq \alpha |x|^2. \)
\end{assumption}

Given a potential $\obj$ satisfying the above assumptions and a suitable cooling schedule $\beta$, we from this point on reserve the notation $\mu_t$ for the Gibbs curve, and $\pi_t$ for the corresponding densities, that is, 
\begin{align}
    \mu_t (\dx) &:= \pi_t(x)\dx, \\
    \pi_t(x) &:= \frac{e^{-\beta(t)\obj(x)}}{\int_\Rdim e^{-\beta(t)\obj(y)}\dy}.
\end{align}
We state now the main result of this section.

\begin{theorem}\label{thm:abs_cont}
    Assume that $\obj\in C^1(\Rdim, \RR)$ satisfies \cref{ass:inf_obj_exists,ass:quadratic_tail,ass:quadratic_gradient}. Let $\beta\in C^1([0,T_1),[1,\infty))$ be a differentiable cooling schedule. Then, for every $T \leq T_1$ such that $\lim_{t\uparrow T} \beta(t) < \infty$, the curve
    \begin{alignat*}{2}
        & \eta \colon & [0,T) &\to \Psp_2(\Rdim) \\
        &              & t &\mapsto \mu_t
    \end{alignat*}
    is absolutely continuous. Further, for every $t\in[0,T)$ there exists a unique velocity field $v_t$ of minimal $L^2(\mu_t)$-norm such that $\mu_t$ and $v_t$ satisfies \eqref{eq:focker_planck_v}.
\end{theorem}

We first approach this theorem by an illustrative example, namely the case when $\obj$ is a Gaussian potential, in the following section. After this, in \cref{sec:v_exists_general_case}, we turn our attention to the general statement of the theorem.

\subsection{The velocity field in the Gaussian case}\label{sec:gaussian_v}

When $\obj$ is a Gaussian potential the analysis is greatly simplified by the existence of explicit solutions to the intermediate transport problems. Let $\obj(x)=\frac{1}{2}x^T\Sigma^{-1}x$, with $\Sigma$ positive definite, and let $\beta\colon[0,T) \to [1,\infty)$ be a smooth non-decreasing cooling schedule. Then, $t\mapsto \mu_t= \cN(0, {\beta(t)}^{-1}\Sigma)$ is a curve of centred Gaussian measures in $(\Psp_2(\Rdim), \wass_2)$ such that the variance tends to $0$ as $\beta\to\infty$.
 For any two Gaussian distributions $\nu_i = \cN(m_i, \Sigma_i)$, $i =1,2$, we have that 
\[
    \wass_2^2(\nu_1,\nu_2) = |m_1-m_2|^2 + \text{Trace}\left(\Sigma_1 + \Sigma_2 - 2\left(\Sigma_1^{1/2}\Sigma_2\Sigma_1^{1/2}\right)^{1/2}\right),
\]
see, e.g., \cite{dowson1982frechet, givens1984class, knott1984optimal, peyre2019computational}.
Additionally, when $\Sigma_1\Sigma_2 = \Sigma_2\Sigma_1$, we have
\[
    \wass_2^2(\nu_1,\nu_2) = |m_1-m_2|^2 + \left\|\Sigma_1^{1/2} - \Sigma_2^{1/2}\right\|_{\text{Fr}}^2,
\]
and therefore
\begin{align*}
    \wass_2^2(\mu_t,\mu_{t+h}) & = \left\|(\beta(t)^{-1/2} - \beta(t+h)^{-1/2}) \Sigma^{1/2}\right\|_{\text{Fr}}^2 \\
                       & = \left|\beta(t)^{-1/2} - \beta(t+h)^{-1/2}\right|^2 \left\|\Sigma^{1/2}\right\|_{\text{Fr}}^2.
\end{align*}
We verify that this is an absolutely continuous curve by verifying that the metric derivative $|\mu'|(t)$ is in $L^1([0,T])$. As $\beta$ is differentiable and non-decreasing, we find that
\begin{align*}
    |\mu'|(t)& = \lim_{h \downarrow  0} h^{-1}\wass_2(\mu_t,\mu_{t+h})                   \\
             & = \lim_{h\downarrow 0} \|\Sigma^{1/2}\|_{\text{Fr}}\, h^{-1}(\beta(t)^{-1/2} - \beta(t+h)^{-1/2}) \\
             & = -\frac{d}{\dt}(\|\Sigma^{1/2}\|_{\text{Fr}}\, \beta(t)^{-1/2}) \\
             & = \frac{\|\Sigma^{1/2}\|_{\text{Fr}}}{2} {\beta'(t)}{\beta(t)^{-3/2}}, 
\end{align*}
which clearly exists for all $t$. Hence
\begin{align*}
    \int_0^T |\mu'|(t) \dt &= \int_0^T -\frac{d}{\dt}(\|\Sigma^{1/2}\|_{\text{Fr}}\,\beta(t)^{-1/2}) \dt \\
    &= \|\Sigma^{1/2}\|_{\text{Fr}}\,(\beta(0)^{-1/2} - \beta(T)^{-1/2} ) < \infty,
\end{align*}
as long as $\beta(0) > 0$. We note that this is finite even when $\lim_{t\to T} \beta(t) = \infty$, and also in the case when $T\to\infty$.

The Monge map $T$ between $\nu_1$ and $\nu_2$ is given by \cite[Remark 2.31]{peyre2019computational}
\[
    T(x) = m_2 + \Sigma_1^{-1/2}\left(\Sigma_1^{1/2}\Sigma_2\Sigma_1^{1/2}\right)^{1/2}\Sigma_1^{-1/2}(x-m_1),
\]
which in the commuting case $\Sigma_1\Sigma_2 = \Sigma_2\Sigma_1$ similarly reduces to \( T(x) = m_2 + \Sigma_2^{1/2}\Sigma_1^{-1/2}(x-m_1). \)
We find that \( T_{\mu_t\to\mu_{t+h}}(x) = \frac{\beta(t+h)^{1/2}}{\beta(t)^{1/2}}x, \)
independently of $\Sigma$. We find by an application of \cref{prop:v_as_monge_limit} that
\begin{align*}    
    v_t(x) &= \lim_{h\to 0} \frac{T_{\mu_t\to\mu_{t+h}}(x) - x}{h} \\
           &= \lim_{h\to 0} \frac{1}{h}\left(\frac{\beta(t+h)^{1/2}}{\beta(t)^{1/2}} - 1\right)x \\
           &= -\frac{1}{2}\frac{\beta'(t)}{\beta(t)}x.
\end{align*}

\begin{remark}
We note that $v_t$ above does not depend on the covariance $\Sigma$. Further, when $\beta'(t)/\beta(t)$ is constant, i.e., $\beta(t)\propto e^{ct}$ then
\[
    v_t(x) = -cx =: \nabla f (x). 
\]
Since this is time-homogeneous, the exponentially cooling Gaussian curve in $\Psp_2$ corresponds to the gradient flow of the convex functional $\cF(\mu) = \int f(x) \dif \mu$ with $f(x) = -c|x|^2/2$, see for instance \cite[Example~2.1]{kazeykina2024ergodicity}.
\end{remark}

\begin{remark}[Unit speed parametrization]
Any absolutely continuous curve permits a unit speed parametrisation \cite[Lemma~1.1.4]{ambrosio2008gradientflows}.
Finding the unit speed curve amounts to solving $|\mu'|(t) = 1$, which in this case reduces to finding a cooling schedule which satisfies
$$
    \beta'(t) = 2a^{-1}\beta(t)^{3/2}, \quad \beta(0) = 1,
$$
with $a := \|\Sigma^{1/2}\|_{\text{Fr}}$. The solution to this differential equation is $\beta(t) = a^2(a-t)^{-2}$, and the curve reaches the endpoint $\delta_0$ at $t=a=\|\Sigma^{1/2}\|_{\text{Fr}}$, i.e., in finite time.
\end{remark}

\subsection{The velocity field in the general case}\label{sec:v_exists_general_case}

This section is devoted to proving \cref{thm:abs_cont}, i.e., showing that the Gibbs curve $t \mapsto \mu_t \in (\sP_2(\Rdim), \wass_2)$ is absolutely continuous. To this end, we first note that due to the Gaussian tails, $\mu_t \in \sP_p$ for any $p$. 

\begin{lemma}\label{lem:gibbs_in_Pp}
    Assume that $\obj$ satisfies \cref{ass:inf_obj_exists,ass:quadratic_tail}. Then, for all $\beta \in(0,\infty)$ and $p\geq 1$, the Gibbs measure satisfies $\mu_\beta \in \Psp_p(\Rdim)$.
\end{lemma}

\begin{proof}
    Let $R$ be as in \eqref{ass:quadratic_tail}. Then $Z_\beta = \int_\Rdim e^{-\beta\obj(x)}\dx \leq D + \int_{|x|>R}e^{-\beta|x|^2} \dx < \infty$ for some constant $D$ and hence $\mu_\beta(\dx) = Z_\beta^{-1}e^{-\beta\obj(x)}\dx$ is a well-defined probability measure. Similarly,
    \[
        \int_\Rdim |x|^p \difempty\mu_\beta(\dx) %
                                  \leq R^p + \frac{1}{Z_\beta}\int_{|x|>R} |x|^pe^{-\beta|x|^2} \dx < \infty. \qedhere%
    \]
\end{proof}

The following lemma establishes a bound on the Poincar\'e constant of $\mu_t$. We give a proof of this by constructing a Lyapunov function and relying on a result in \cite{bakry2008simple}. For details, see \cref{app:proofs}.

\begin{restatable}%
{lemma}{poincarelemma}\label{lem:poinc_up_bound}
    Let $\obj$ satisfy \cref{ass:inf_obj_exists,ass:quadratic_gradient,ass:quadratic_tail} and let $\beta\colon[0,T]\to(0,\infty)$. Then, for every $t$, $\mu_t$ satisfies a Poincar\'e inequality:
    \[
        \textnormal{Var}_{\mu_t}(f) =  \int \left( f - \int f \dif\mu_t \right)^2\dif\mu_t \leq \poinc(\mu_t) \int |\nabla f |^2\dif\mu_t,
    \]
    for all $f \in C^\infty_c(\Rdim)$. Further, there exist positive constants $A,B,K$ such that, for all $t\in[0,T]$
    \[
        \poinc(\mu_t) \leq C_t :=A(1+Be^{\beta(t)K}).
    \]
\end{restatable}

Since the Poincar\'e constant $\poinc(\mu_t)$ is bounded by the preceding lemma, the Sobolev space $H^1_t := H^1(\mu_t)$ is well-defined.
Next, note that for the Gibbs curve, the continuity equation for the density $\pi_t$ reads
\[
    \partial_t\pi_t(x) + \nabla \cdot(v_t(x)\pi_t(x)) = 0.
\]
Differentiating $\pi_t$ with respect to $t$, we find that
\begin{equation}\label{eq:dpidt}
    \partial_t\pi_t(x) = -\beta'(t)\left(\obj(x)-\int \obj \dif \mu_t\right)\pi_t(x).
\end{equation}
As a brief aside, we note that $\int f(x) \difempty\partial_t\pi_t(x)\dx = -\beta'(t)\,\covt(f,\obj)$, and hence we can understand $\partial_t\mu_t$ as measuring the covariance of a function $f$ against the potential $\obj$. By \eqref{eq:dpidt} we arrive at the following equation for $v_t$
\begin{equation}
\label{eq:gibbs_continuity_equation}
    \nabla\cdot (v_t(x)\pi_t(x))  = \beta'(t)\left(U(x) - \int U \dif \mu_t\right)\pi_t(x).
\end{equation}
In the proof of the following theorem, we will look for a solution to \eqref{eq:gibbs_continuity_equation} in the space of gradients of functions in $H^1$. To this end, as $\mu_t$ satisfies a Poincar\'e inequality the bilinear map $\langle\cdot,\cdot\rangle_{\nabla \mu_t}\colon H_t^1 \times H_t^1 \to \RR$, given by $\langle\cdot,\cdot\rangle_{\nabla \mu_t} := \vdot{\nabla \cdot}{\nabla \cdot}_{\mu_t}$, defines an inner product on the subspace of zero mean functions
\[
    \dot H_t^1 = \left\{f \in H^1_t \colon \int f \dif \mu_t = 0\right\}.
\]
To see this, first note that
\[
    \langle f, g \rangle_{\nabla \mu_t} = \int \langle \nabla f, \nabla g\rangle \dif \mu_t
\]
is well-defined by Cauchy--Schwarz. Moreover, we have that \( \langle f,f \rangle_{\nabla \mu_t} \geq 0, \) and by the Poincar\'e inequality, equality holds if and only if $f = 0$. We can now prove \cref{thm:abs_cont} in the following form:
\begin{theorem}\label{thm:gibbs_ac}
    Assume that $\beta\colon[0,T) \to [1,\infty)$ is differentiable and that $\obj$ satisfies \cref{ass:inf_obj_exists,ass:quadratic_tail,ass:quadratic_gradient}.
    Then, for each $t \in[0,T)$ there exists a unique $h_t \in \dot H^1_t$ such that the velocity field $v_t  = \nabla h_t \in L^2(\mu_t, \Rdim, \Rdim)$ solves the continuity equation
    \[
        -\nabla\cdot(v_t\mu_t) = \partial_t \mu_t
    \]
    in the sense that
    \begin{equation}\label{eq:v_t_weak_form}
        \int \langle v_t(x), \nabla f(x) \rangle\, \pi_t(x)\dx = \int f(x) \difempty\partial_t\pi_t(x)\dx, \quad \forall f \in C^\infty_c(\Rdim). 
    \end{equation}
    This $v_t$ is the vector field with minimal $\| \cdot \|_{L^2(\mu_t)}$-norm of all vector fields that satisfies \eqref{eq:v_t_weak_form}.
    Further, $\int_0^T \|v_t\|_{L^2(\mu_t)} \dt < \infty$, and hence the Gibbs curve is absolutely continuous.
\end{theorem}

\begin{proof}
    Consider the linear functional $\phi\colon \dot H^1_t\to\RR$ given by
    \begin{align}\label{eq:phi}
        \phi(f) &= \int f(x) \difempty\partial_t\pi_t(x) \dx = -\beta'(t) \int f(x)\left(\obj(x)-\int \obj d\mu_t\right)\pi_t(x) \dx.
    \end{align}
    By applying H\"{o}lder's inequality (Cauchy--Schwarz) we find that
    \begin{align*}
        |\phi(f)| &\leq \left|\beta'(t)\right| \left(\int |f(x)|^2 \pi_t(x) \dx \right)^{1/2} \left(\int\left(\obj(x)-\int \obj \dif \mu_t\right)^2\pi_t(x) \dx\right)^{1/2}\\
            &= |\beta'(t)|\sqrt{\var_{\mu_t} f }\sqrt{\var_{\mu_t}\obj},
    \end{align*}
    In conjunction with the bound $C_t$ on the Poincar\'e constant of $\mu_t$ in \cref{lem:poinc_up_bound} we find then the following estimate for the dual norm of $\phi$:
        \begin{align*}
        \| \phi \|_{(\dot H^1_t)^*} &= \sup_{f\in \dot H^1_t} \frac{\left|\phi(f)\right|}{\sqrt{\langle \nabla f, \nabla f\rangle_{\mu_t}}} \\
         &\leq  \sup_{f\in \dot H^1_t} \frac{|\beta'(t)|\sqrt{\var_{\mu_t} f }\sqrt{\var_{\mu_t}\obj}}{\sqrt{\langle \nabla f, \nabla f\rangle_{\mu_t}}} \\
         &= |\beta'(t)| \sqrt{C_P(\mu_t)} \sqrt{\vart \obj} \\
         &\leq |\beta'(t)|\sqrt{C_t}\sqrt{\vart\obj}.
    \end{align*}

    Since $\sqrt{\vart \obj}$ is bounded (\cref{lem:u_bounded_variance}), the linear functional $\phi$ is continuous. By the Riesz representation theorem \cite[p.~109]{luenberger1969optimization} there exists a unique element $h_t$ in the Hilbert space $(\dot H^1_t, \langle \cdot, \cdot\rangle_{\nabla \mu_t})$ such that
    \[
        \phi(f) = \langle h_t, f \rangle_{\nabla\mu_t} \quad \text{and} \quad \|\phi\|_{(\dot H^1_t)^*} = \|h_t\|_{\dot H^1_t}.
    \]
    From the first of these equations we find in particular that, for all $f \in C^\infty_c \subset \dot H^1_t$,
    \[
        \int f(x) \difempty\partial_t\pi_t(x) \dx = \int \langle \nabla h_t(x), \nabla f(x)\rangle \,\pi_t(x) \dx,
    \]
    so $v_t := \nabla h_t$ solves the continuity equation in the prescribed sense, \eqref{eq:v_t_weak_form}. Further,
    \[
        \|v_t\|_{L^2(\mu_t)} = \|\nabla h_t\|_{L^2(\mu_t)} = \|h_t\|_{\dot H^1_t} = \|\phi\|_{(\dot H^1_t)^*} \leq 
         |\beta'(t)|\sqrt{C_t}\sqrt{\var_{\mu_t}\obj}.
    \]
    Since $\sup_{t\in[0,T]} |\beta'(t)|\sqrt{C_t}\sqrt{\var_{\mu_t}\obj} < \infty$, $\|v_t\|_{L^2(\mu_t)} \in L^1([0,T])$ and by \cite[Theorem~8.3.1]{ambrosio2008gradientflows}, phrased as Theorem~\ref{thm:ac_curves_and_velocity_fields}, the Gibbs curve is absolutely continuous.

    Finally, the fact that the solution $v_t$ is the element with minimal $\| \cdot \|_{L^2(\mu_t)}$-norm of all solutions to \eqref{eq:v_t_weak_form} follows from \cite[Proposition~8.4.5]{ambrosio2008gradientflows}, by noting that
\[
v_t = \nabla h_t \in \overline{\{ \nabla \varphi \colon \varphi \in C^\infty_c(\Rdim)  \}}^{L^2(\mu_t)}, %
\]
where $\overline{\phantom{A}}^{L^2(\mu_t)}$ denotes the closure in the $L^2(\mu_t)$-norm.
\end{proof}

\begin{remark}
In the preceding proof, it is interesting to note that the ansatz $v_t = \nabla h_t$ turns the continuity equation \eqref{eq:gibbs_continuity_equation} into an elliptic PDE.
To this end, let $\cL_t$ be the operator defined by
\[
    \cL_tf := \frac{1}{\beta(t)}\Delta f - \langle \nabla\obj, \nabla f\rangle.
\]
As $\pi_t(x) > 0$ for all $x$, this allows us to rewrite \eqref{eq:gibbs_continuity_equation} as
\begin{equation}\label{eq:Lh=pttpit}
    \cL_th_t(x) = \frac{\beta'(t)}{\beta(t)}\left(\obj(x)-\int\obj \dif \mu_t\right) =: g_t(x).
\end{equation}

Assuming that the preceding equation permits a sufficiently regular (classical) solution $h_t$, It\^o's formula yields a probabilistic representation of $h_t$ by
\begin{align}\label{eq:probabilistic_ht}
    h_{t}(x) = \int h_t \dif \mu_t -\lim_{S\rightarrow \infty}\mathbb{E}\left[\int_{0}^{S} g_{t}(Y^x_s) \dif s\right], %
\end{align}
where %
$Y^x_s$ is a solution to the overdamped Langevin SDE at a fixed temperature $\beta(t)$ starting at $x$:
\begin{align*}%
    \dif Y^x_s= - \nabla \obj(Y^x_s) \dif s + \sqrt{\frac{2}{\beta(t)}} \dif B_s, \quad Y^x_0 = x.
\end{align*}
Indeed, an application of It\^o's formula yields that   %
\begin{align*}
    h_{t}(Y^x_S) = h_{t}(x) + \int_{0}^{S} \cL_t h_{t}(Y_s^x) \dif s + \frac{2}{\beta(t)}\int_{0}^{S} \nabla \cdot  h_{t}(Y_s^x) \dif B_s.
\end{align*}
We then find this probabilistic interpretation by first taking the expected value of both sides
\begin{align*}%
    \EE \left[ h_{t}(Y^x_S) \right] = h_{t}(x) + \EE \left[\int_{0}^{S} \cL_t h_{t}(Y^x_s) \dif s \right],
\end{align*}
and then noting that, under \cref{ass:inf_obj_exists,ass:quadratic_tail,ass:quadratic_gradient}, the Markov process $Y^x_s$ is \emph{geometrically ergodic}, and specifically it holds that $\mathbb{E} \left[h_{t}(Y^x_S)\right] \rightarrow  \int h_t\dif\mu_t $ as $S \rightarrow \infty$, for suitably regular $h_t$.
Clearly, $\int g_t \dif\mu_t=0$, and again, by the geometric ergodicity of $Y^x_s$ we find that $|\mathbb{E}\left[g_{t}(Y_s^x)\right]| \leq c_1 e^{-c_2 s}$, for some $c_1, c_2 >0$ independent of $s$. Hence, as the right-hand side of \eqref{eq:probabilistic_ht} is finite, the claim follows. %
\end{remark}

\subsection{On finiteness of \texorpdfstring{$\beta(T)$}{β(T)} and an exponentially bad Poincar\'e constant}\label{sec:finite_end}

We elaborate here on the restriction of finite end inverse temperatures $\beta(T)$. First, we consider the strength of the previous results in the Gaussian setting of Section \ref{sec:gaussian_v}.

\begin{example}\label{ex:gauss_unit_speed}
    Picking again $\obj(x)=|x|^2/2$ so that $\mu_t = \cN(0,\beta(t)^{-1}I_\ddim)$. Assuming that $\beta' \geq 0$,
    we now investigate the curve length bound
    \[
        \ell \leq \int_0^T \left(\beta'(t)\sqrt{C_t}\sqrt{\vart(U)}\right) \dt.
    \]
    A straightforward computation yields
    \begin{align*}
        \vart(U) &= \vart(|x|^2/2) \\  
         & = \frac{1}{4}\textnormal{Var}\left(\sum_{i=1}^\ddim \beta(t)^{-1}Y_i^2\right) \\
                                              & = \frac{1}{4}\frac{1}{\beta(t)^2}2\ddim,
    \end{align*}
    where we denoted by $Y_i$ independent standard Gaussian variables on $\RR$.
    As \( \sqrt{\vart(|x|^2/2)} = \sqrt{\frac{\ddim}{2}}\frac{1}{\beta(t)} \) we have that
    \[
        \ell \leq \frac{\ddim}{2} \int_0^T\left( \frac{\beta'(t)}{\beta(t)}\sqrt{C_t}\right) \dt,
    \]    
    from which we can see that if $\beta$ diverges, one needs additional control on the Poincar\'e constant to ensure that this bound guarantees integrability. In this case, it is known that the optimal Poincar\'e constant is $ \poinc \propto \frac{1}{\sqrt{\beta}}$,
    which yields that, for some constant $C > 0$,
    \[
        \ell \leq C \int_0^T \beta'(t)\beta(t)^{-5/4} \dt = 4C(\beta(0)^{-1/4} - \beta(T)^{-1/4}),
    \]
    which converges even if $\lim_{t\to T}\beta(t) = \infty$.
\end{example}

As we just saw in the Gaussian case, the curve has finite length even as $\beta\to\infty$. This would in general be desirable; however, we note that for potentials satisfying \cref{ass:inf_obj_exists,ass:quadratic_tail,ass:quadratic_gradient}, the current proof technique cannot show this. In the proof of \cref{thm:gibbs_ac}, we bound the curve length using the Poincar\'e inequality in conjunction with the exponential bound on the Poincar\'e constant in \cref{lem:poinc_up_bound}. In the next example, we show that there exist cases where the exponential bound on the Poincar\'e constant is not overly pessimistic. That is, in general, as the inverse temperature grows large, the Poincar\'e constant can get exponentially bad.

\begin{example}[An exponentially bad Poincar\'e constant for a potential with a single global minima]

Consider now a probability measure $\mu(\dx) = \pi(x)\dx$ on $\RR$, with median $m$ such that $\mu([m,+\infty))\geq \frac{1}{2}$ and $\mu((-\infty,m])\geq \frac{1}{2}$. By \emph{Muckenhoupt's criterion}, we have the following lower bound on the Poincar\'e constant of $\mu$:
\[
  \poinc(\mu) \geq \frac{1}{2}\sup_{y\colon y > m} \mu([y,+\infty))\int_m^y\frac{1}{\pi(x)} \dif x,
\]
see \cite[Theorem~4.5.1]{bakry2014analysis}. Further, let $I=(a,b) \subset \RR$ be an open interval with $\mu(I) > 0$, and let $\mu_{\mid I}$ be the restriction of $\mu$ to $I$ defined by \( \mu_{\mid I}(A) = \mu(A \cap I) / \mu(I) \). Then $\poinc(\mu_{|I}) \leq \poinc(\mu)$, that is, restricting the measure to a sub-interval can not make the Poincar\'e inequality weaker.
    
Following Miclo \cite{miclo2008poinc}, let $\obj$ be an even potential on $\RR$. Fix an inverse temperature $\beta$ and let $\mu = \frac{1}{Z}e^{-\beta \obj(x)}\dx$. Since $\obj$ is even $\mu$ has median $m=0$. Relabel and renormalize it to a restriction $\mu$ on some symmetric interval $(-a,a)$.
By Muckenhoupt's criterion we have that
\begin{align*}
    2\poinc(\mu) & \geq \sup_{y\colon 0<y<a} \mu([y,a))\int_0^y\frac{1}{\pi(x)} \dx                                                                               \\
            & = \sup_{y\colon 0<y<a} \left[\int_y^ae^{-\beta\obj(x)}\dx\right]\left[\int_0^ye^{\beta\obj(x)} \dx\right],                           
\intertext{and by Jensen's inequality we find that}
        2\poinc(\mu) & \geq \sup_{y\colon 0<y<a} (a-y)\left[e^{-\frac{1}{a-y}\int_y^a\beta\obj(x)\dx}\right]\cdot y\left[e^{\frac{1}{y}\int_0^y\beta\obj(x)\dx}\right] \\
            & = \sup_{y\colon 0<y<a} y(a-y)\exp\left[\beta\left(\frac{1}{y}\int_0^y\obj(x) \dx-\frac{1}{a-y}\int_y^a\obj(x) \dx \right)\right].                   
\end{align*}
By the restriction property it is thus sufficient to design an even function such that there is a pair $0<y<a$ where $\obj$ is on (Lebesgue-)average greater on the interval $(0,y)$ than on $(y,a)$, i.e., such that
\[
    \frac{1}{y}\int_0^y\obj(x) \dx > \frac{1}{a-y}\int_y^a\obj(x) \dx.
\]
Then the pair $(y,a)$ gives a lower bound on the supremum and $e^{c\beta} \lesssim \poinc(\mu_\beta)$.
Led by this, we can be constructive and pick \( \obj(x) = x^2 - 5\cos(\pi x)+5 \), for which we find that
\[
    \frac{1}{1.7}\int_0^{1.7}\obj(x) \dx \approx 1.7 \quad \text{and} \quad
    \frac{1}{0.3} \int_{1.7}^{2.1} \obj(x) \dx \approx -1.1.
\]
We see then that $c\approx 3$ and the Poincar\'e constant is lower bounded by
$\exp(3\beta) \lesssim \poinc(\mu_\beta)$. To conclude, we note that this choice of $\obj$ satisfies \cref{ass:inf_obj_exists,ass:quadratic_tail,ass:quadratic_gradient}.
\end{example}

\section{Controlled stochastic annealing via superposition}\label{sec:controlled_superposition}

The goal of this section is to establish that there exists stochastic processes driven by both the velocity field $v_t$ and additional stochastic dynamics. Throughout this section we assume that $\obj$ satisfies Assumptions \ref{ass:inf_obj_exists}, \ref{ass:quadratic_tail}, and \ref{ass:quadratic_gradient}. In order to simplify the arguments, we assume that the tail behaviour of the velocity fields is similar to the ones we observed in the Gaussian example.

\begin{assumption}[Linear growth on $v_t$]\label{ass:v_linear_growth}
    There exists constants $c_1, c_2$ such that,
    for all $t > 0$, $x \in \Rdim$,
    \[
        |v_t(x)| \leq c_1 |x| + c_2.
    \]
    
\end{assumption}

\subsection{Controlled simulated annealing with diffusive dynamics}\label{sec:fss_diffusive}

We will begin by investigating the case where particles are driven by both the velocity field and the diffusive Langevin-like dynamics \eqref{eq:langevin}, that is, the SDE  
\begin{equation}\label{eq:SDE_v_plus_langevin}
    \dif X_t = \left(v_t(X_t) - \nabla U ( X_t)\right) \dt + \sqrt{\frac{2}{\beta(t)}} \dif B_t, \quad X_0 \sim \mu_0.   
\end{equation}
We show that such a process exists, and further, that the law of this process does not diverge from the Gibbs curve $\mu_t$. 
It is instructive to first sketch the argument, postponing some of the technicalities to later.
If it exists, the evolution of the law $\rho_t$ of $X_t$ in \eqref{eq:SDE_v_plus_langevin} is given by the Fokker--Planck equation
\begin{equation}\label{eq:FP_SDE_v_plus_langevin}
    \partial_t \rho_t + \nabla\cdot(v_t\rho_t) + \nabla\cdot(-\nabla U\rho_t) + \frac{1}{\beta(t)}\Delta\rho_t = 0, \quad \rho_0 = \mu_0,
\end{equation}
and we want to show that the Gibbs curve $\mu_t$ is the unique solution. To see that this is the case, recall that $v_t$ and $\mu_t$ jointly satisfy the continuity equation
\[
    \partial_t \mu_t + \nabla\cdot(v_t\mu_t) = 0.
\]
The other two terms, $\nabla\cdot(-\nabla U\rho_t) + \frac{1}{\beta(t)}\Delta\rho_t$, correspond to the $\mu_t$-stationary diffusive dynamics $-\nabla\obj(X_s) \dif s + \sqrt{2\beta(t)^{-1}} \dif B_s $ and hence
\[
    \nabla\cdot(-\nabla U\mu_t) + \frac{1}{\beta(t)}\Delta\mu_t = 0.
\]
Putting this together we find that, at least in a formal sense, the Gibbs curve solves \eqref{eq:FP_SDE_v_plus_langevin}, that is,
\begin{equation*}\label{eq:gibbs_curve_fokker_planck_vsde}
    \partial_t \mu_t + \nabla\cdot(v_t\mu_t) + \nabla\cdot(\nabla U\mu_t) + \frac{1}{\beta(t)}\Delta\mu_t = 0.    
\end{equation*}

We justify the preceding computations via the following two propositions. First, by showing that the SDE \eqref{eq:SDE_v_plus_langevin} permits a solution, \cref{prop:weak_sol_SDE_plus_v} justifies the Fokker--Planck equation \eqref{eq:FP_SDE_v_plus_langevin}. Then, in \cref{prop:gibbs_curve_sol_fp_vsde}, we show that 
this solution is unique and hence necessarily agrees with the Gibbs curve.

\begin{proposition}\label{prop:weak_sol_SDE_plus_v}
    Let $v_t$ be as in \cref{thm:gibbs_ac} and assume that $v_t$ satisfies \cref{ass:v_linear_growth}. Then the stochastic differential equation \eqref{eq:SDE_v_plus_langevin} has a weak solution $(X_t)_{t\geq 0}$.
\end{proposition}

\begin{proof}
    Define formally the operator $\modgen_t$ acting on test functions $f$ by
    \[
        \modgen_tf = \langle v_t, f \rangle + \cL_t f = \langle v_t, f \rangle + \langle - \nabla U, \nabla f \rangle+  \frac{1}{\beta(t)} \Delta f.
    \]
    The \emph{martingale problem} for $\modgen_t$ is well-posed if there, for all $x, t\geq \tau$, exists a unique probability measure $P_x$ on the path space $C([0,T],\Rdim)$ such that $P_x[X_s = x, 0 \leq s \leq \tau] = 1$ and
    \[
        f(X_t) - f(x) - \int_\tau^t (\modgen_s f)(X_s) \dif s 
    \]
    is a martingale.
    \cref{ass:quadratic_gradient} and \cref{ass:v_linear_growth} yield, respectively,
    \[
        \langle -\nabla \obj(x), x\rangle \leq 0, \quad \text{and} \quad \langle v_t(x),x\rangle \leq C(1+|x|^2),
    \]
    for sufficiently large $|x|$. Hence the drift $b(t,x) = -\nabla\obj(x) + v_t(x)$ satisfies the conditions of \cite[{Theorem~10.2.2}]{stroock1997multidimensional}, namely
    \[
        \langle x, b(t,x) \rangle \leq C(1+|x|^2),
    \]
    which guarantees the existence of such a $P_x$. Further, by \cite[{Theorem~32.7}]{kallenberg2021foundations}, this is a weak solution to \eqref{eq:SDE_v_plus_langevin}. (See also \cite[p.~170]{rogers2000diffusions}).
\end{proof}

The above statement tells us that, for any initial point $X_0 = x_0$ there is a stochastic process $(X_t)_{t\geq 0}$ solving the SDE \eqref{eq:SDE_v_plus_langevin}. We now want to ensure that this has the prescribed time marginals. With $(\modgen_t)^*$ the adjoint of $\modgen_t$, we reformulate the Fokker--Planck Equation \eqref{eq:FP_SDE_v_plus_langevin} by
\begin{equation}\label{eq:MFP_v_plus_langevin}
    \partial_t \nu_t = (\modgen_t)^*\nu_t, \quad \nu_0 = \mu_0,
\end{equation}
by which we mean that a family of measures $(\nu_t)$ on $\Rdim$ solves \eqref{eq:MFP_v_plus_langevin} if and only if, for every $t$,
\begin{equation}\label{eq:weak_form_kolmogorov_forward}
    \int_\Rdim f(t,x) \difempty\nu_t(\dx) - \int_\Rdim f(0, x) \difempty\mu_0(\dx) = \int_0^t \int_{\Rdim} \left( \partial_s f(s,x) + \modgen_s f(s,x) \right)  \difempty\nu_s(\dx) \dif s
\end{equation}
for all \( f \in C_b^{1,2}([0,T] \times \Rdim) \) where $C_b^{1,2}([0,T] \times \Rdim)
:= \{ f(t,x) \colon [0,T]\times \Rdim \to \RR \mid f, \partial_t f, \partial_i f, \partial_i\partial_j f \in C_b([0,T] \times \Rdim)\}.$
Note that, by construction, the Gibbs curve $\mu_t$ solves \eqref{eq:MFP_v_plus_langevin}.

\begin{proposition}\label{prop:gibbs_curve_sol_fp_vsde}
    Let $(X_t)_{t\geq 0}$ be a weak solution to \eqref{eq:SDE_v_plus_langevin} with $X_0 \sim \mu_0$. Then
    \[
        \textnormal{Law}(X_t) = \mu_t \quad \text{for almost every } t \in [0,T].
    \]
\end{proposition}

\begin{proof}
    By \cite[{Theorem~32.10}]{kallenberg2021foundations} the measure $P_{\mu_0} = \int P_x \mu_0(\dx)$ exists and is unique, and hence the law of $X_t$ with $X_0\sim\mu_0$ is unique. Write $\rho_t := \text{Law}(X_t)$ and denote by \( b(t,x) = -\nabla\obj(x) + v_t(x)\).

    Let $f \in C_b^{1,2}([0,T] \times \Rdim)$. By It\^o's formula
    \begin{align*}
        \dif f(t, X_t) &= \partial_t f(t,X_t) \dt + \langle\nabla_x f(t,X_t), \dif X_t\rangle + \frac{1}{\beta(t)}\Delta_x f(t,X_t) \dt \\
            &= \partial_t f(t,X_t) \dt + \langle\nabla_x f(t,X_t) ,b(t,X_t)\rangle \dt + \frac{1}{\beta(t)} \nabla_x \cdot f(t,X_t) \dif B_t + \frac{1}{\beta(t)} \Delta_x f(t,X_t) \dt.
    \end{align*}

    Since $\partial_i f$ is bounded, $\EE \left[ \int_0^T |\nabla_x \cdot f(t,X_t) \beta^{-1}(t)|^2 \dt \right] < \infty$, so $\EE \left[ \int_0^t \beta^{-1}(s)\nabla_x \cdot f(s,X_s)  \dif B_s \right] = 0$. Now
    \begin{align*}
        \EE[f(t,X_t)] - \EE[f(0,X_0)] &= \EE \left[\int_0^t \partial_s f(s,X_s) \dif s\right] + 
            \EE \left[\int_0^t \langle \nabla_x f(s,X_s),  b(s,X_s)\rangle \dif s \right] \\
            &+ \EE \left[\int_0^t \frac{1}{\beta(s)}\Delta_x f(s,X_s) \dif s \right].
    \end{align*}
    Identifying that $\EE[f(t,X_t)] = \int f(t,x)\rho_t(\dx)$, and similarly for the other terms, we find that
    \[
        \int_\Rdim f(t,x) \rho_t(\dx) - \int_\Rdim f(0, x) \mu_0(\dx) = \int_0^t \int_{\Rdim} \left( \partial_s f(s,x)  +  \modgen_s f(s,x) \right) \rho_s(\dx) \dif s, 
    \]
    and in fact, by \cite[Theorem~9.3.6]{bogachev2022fokker} the solution $\rho = \rho_t(\dx) \dt$ is unique,%
    \footnote{This is verified by noting that $v_\cdot \in L^1(\mu_t(\dx) \dt, \Rdim \times (0,T))$, since we have already shown that $v$ is more regular, $\|v_t\|_{L^2(\mu_t)} \in L^1(\dt, [0,T])$.} %
    and hence we conclude that the Gibbs space-time measure agrees with the law of the process, i.e., $\mu = \mu_t(\dx)\dt = \text{Law}(X_t)\dt$.
\end{proof}

\subsection{Superposition with other stochastic dynamics}

As alluded to in the introduction, we now revisit the discussion at the beginning of the previous section, in light of the Kolmogorov forward equation \eqref{eq:MFP_v_plus_langevin}. Starting this time from the continuity equation,
\[
    \partial_t \mu_t  = - \nabla \cdot(v_t\mu_t),
\]
we have seen that with $\cL_t$ denoting the diffusion generator, also
\begin{equation}
    \partial_t \mu_t =  - \nabla \cdot(v_t\mu_t) + \cL_t^*\mu_t.
\end{equation}
At least formally, this is clear since $\cL_t$ preserves $\mu_t$ in the sense that, for each fixed $t$, $\cL_t^* \mu_t = 0$. In other words, we can think of the family $(\cL_t)$ as a homogeneous solution to the continuity equation. Additionally, we saw that there indeed is a Markov process $(X_t)$ with generator $\gen^o_tf = \langle v_t, \nabla f\rangle + \cL_tf $. It is tempting to then consider 
\[
    \modgen_tf = \langle v_t,\nabla f\rangle + \gen_tf
\]
where, for each $t$, $\gen_t$ is the generator \emph{any} $\mu_t$-stationary stochastic process, i.e.~satisfying $\gen_t^* \mu_t = 0$, since then formally
\begin{align}\label{eq:transport_kolmogorov_equation}
    \partial_t\mu_t &= -\nabla\cdot (v_t\mu_t) + \cA^*_t\mu_t.    
\end{align}
However, it is in general not trivial to show that the right-hand side of \eqref{eq:transport_kolmogorov_equation} defines a generator for a Markov process. Nevertheless, we next give another example of where this is indeed the case.

\subsubsection{Controlled piecewise deterministic simulated annealing}\label{sec:fss_pdmp}

We will now turn our attention to a case of non-diffusive stochastic dynamics. 
The Bouncy Particle Sampler (BPS) is a rejection-free Monte Carlo method introduced in \cite{peters2012rejection}; see also \cite{bouchard2018bouncy}. Samples from a distribution of interest, \( \mu(\dx) = \dens(x) \dx \propto \exp(-\obj(x)) \dx, \)
are produced by simulating a particular \emph{piecewise deterministic Markov process} (PDMP). For an in-depth treatment on PDMPs, see \cite{davis1993markov}. If $\dens$ is defined on $\Rdim$, then the BPS
process $(Z_t)_{t\geq 0} = ((X_t,\bpsvel_t))_{t \geq 0}$ takes values in an extended state space $\Rdim\times S^{\ddim-1}$, where $S^{\ddim-1} \subset \Rdim$ denotes the unit sphere. We think of $X_t$ as a position process, and $Y_t$ as a velocity process. In between events, the position $X_t$ follows linear trajectories in the piecewise constant velocity $Y_t$,
\begin{align*}
    \frac{\dif}{\dt} \begin{pmatrix}
        X_t \\
        Y_t
    \end{pmatrix} =  \begin{pmatrix}
        \bpsvel_t \\
        0
    \end{pmatrix}.
\end{align*}
Clearly, this ODE is solved by the smooth map $\xi_t(x,y) = (x + ty, y)$.
At reflection events, arriving according to an inhomogeneous Poisson process with the instantaneous rate given by
\begin{equation}
    \lambda(X_t, \bpsvel_t) = \max\left\{0, \langle \bpsvel_t, \nabla \obj(X_t) \rangle \right\},
\end{equation}
the velocity changes deterministically, given by a specular reflection against the level set of $\obj$ in the current point. That is, for a process in state $X_t$ and with velocity $\bpsvel_t$ at a jump time, the velocity is changed to
\begin{equation}\label{eq:velocity_reflection}
    \bpsvel'_t = R(X_t)\bpsvel_t, \quad \text{where} \, R(x) := I - 2\frac{{\nabla\obj(x)}{\nabla\obj(x)}^T}{|{\nabla\obj(x)}|^2}
\end{equation}
and $R(x) = I$ when the latter is not defined.
With $\psi(\dy)$ the uniform measure on $S^{\ddim - 1}$, this process permits $\pi(x)\dx\difempty\psi(\dy)$ as a stationary measure. However, this measure might in general not be the only stationary measure if one does not impose additional velocity refreshment events \cite{bouchard2018bouncy}. At these events, the velocity $Y_t$ is resampled from the velocity distribution. To achieve this, an additional jump kernel, with a constant rate of events $\lambda_R$, is superimposed. The generator of the process is then given by
\begin{align*}
    \gen f(x,y) & = \langle \nabla_xf(x,y), y\rangle                         \\
                  & \phantom{+}+ \lambda(x,y)\left(f(x,R(x)y) - f(x,y)\right) \\
                  & \phantom{+}+ \lambda_R \int (f(x,y')-f(x,y))\psi(\dif y'),
\end{align*}
satisfying 
\begin{equation}\label{eq:bps_invariance}
    \iint \gen f(x,y) \exp(-\obj(x))\dx\difempty\psi(\dy) = 0
\end{equation}
for a suitable class of test functions.

\paragraph{Piecewise deterministic simulated annealing}
In \cite{monmarche2016piecewise} the BPS process is used as the basis for the simulated annealing process discussed in the previous section.
Here, a theoretical analysis of the maximal cooling rate such that the law of the $X_t$-marginal converges to $\mu_\infty$ reveals that this process, as in the diffusion case, requires a logarithmic cooling schedule for general $\obj$. For a cooling schedule $\beta$, the time-inhomogeneous BPS
is the Feller process on $\Rdim \times S^{\ddim-1}$ with generator given by
\begin{align*}
    \gen_t f(x,y) & = \langle \nabla_xf(x,y), y\rangle                         \\
                  & \phantom{+}+ \beta(t)\lambda(x,y)\left(f(x,R(x)y) - f(x,y)\right) \\
                  & \phantom{+}+ \lambda_R(t) \int (f(x,y')-f(x,y))\psi(\dif y')
\end{align*}
where $\psi = \text{Unif}(S^{d-1})$ and $\lambda_R(t)\geq 0$ is some possibly time-dependent refreshment rate. By \eqref{eq:bps_invariance}, $\gen_t$ is $\mu_t$-invariant.

\newcommand{\refrate}{\lambda_\textrm{R}}

\paragraph{Controlled piecewise deterministic simulated annealing} We consider now the modified BPS on $(\Rdim, S^{\ddim-1})$ driven by the deterministic dynamics
\begin{align*}
    \frac{\dif}{\dt} \begin{pmatrix}
        X_t \\
        Y_t
    \end{pmatrix} =  \begin{pmatrix}
        \bpsvel_t + v_t(X_t) \\
        0
    \end{pmatrix}.
\end{align*}
To specify a PDMP with these characteristics, we need to be able to integrate these dynamics. To this end, we add an additional regularity assumption on the velocity field $v_t$.
\begin{assumption}\label{ass:lipschitz_v}
    Assume $v_t(x)$ is continuous in time and Lipschitz in space, in the sense that there exists an $L > 0$, independent of $t$, such that
    \[
        |v_t(x)-v_t(y)| \leq L|x-y| \quad \forall x,y,t.
    \]
\end{assumption}
The preceding assumption, strictly stronger than \cref{ass:v_linear_growth}, guarantees that, for all $s,x,y$, the initial value problem
\begin{equation}\label{eq:inhom_v_flow}
    \frac{\dif{}}{\dt}\zeta^s_t(x,y) = v_t(\zeta^s_t(x,y)) + y, \quad \zeta^s_0(x,y) = x
\end{equation}
has a global (in time) solution $\zeta$. For suitable functions $f$ in the domain of the generator $\modgen_t$ of this process at time $t$, we have that
\[
    \gen^o_tf(x,y) = \langle\nabla_x f(x,y), v_t(x)\rangle + \gen_t f(x,y),
\]
and as $\gen_t$ preserves $\mu_t$, we find that
\begin{equation}\label{eq:modpdmp_solves_fp}
    \iint \modgen_tf(x,y) \pi_t(x)\dx \difempty\psi(\dy) = \iint f(x,y) \partial_t \pi_t(x)\dx \difempty\psi(\dy).
\end{equation}

Under the conditions introduced so far, the following statement establishes results analogous to those in \cref{sec:fss_diffusive}.
\begin{proposition}\label{prop:mbps_works}
    Assume that $\obj$ satisfies \cref{ass:inf_obj_exists,ass:quadratic_tail,ass:quadratic_gradient}. Let $v_t$ be as in \cref{thm:gibbs_ac} and assume that it further satisfies \cref{ass:lipschitz_v}. Then there exists a PDMP $Z_t = (X_t,Y_t)$, taking values in $\Rdim \times S^{\ddim -1}$, with extended generator $\cA^o_t$ such that, with $Z_0 \sim \pi_0(x) \dx \difempty\psi(\dy)$, it holds that
    \[
        \textnormal{Law}(Z_t) = \pi_t(x)\dx\difempty\psi(\dy),%
    \]
    where $\psi(\dy)$ is the uniform measure on the sphere $S^{\ddim -1}$.
    
\end{proposition}
 
\begin{proof}
    Adding time as a coordinate, we introduce a homogeneous space-time PDMP $\cZ_t = (t, X_t, \bpsvel_t)$ taking values in $\cM := [0,T] \times \Rdim \times S^{\ddim - 1}$. This process follows the homogeneous flow
    \[
        \xi_t(s,X_s,Y_s) = (s+t, \zeta_t^s(x,y), 0), \quad \text{for $s,t \geq 0, s+t < T$}
    \]
    until it reaches the boundary $s=T$, at which point it stops. By \cref{ass:lipschitz_v}, $\xi_t$ is as regular as need be. From \cite[Proposition~4.2]{holderrieth2021cores}, we can verify that this process is Feller by showing that the flow $\zeta$ induced by $v_t$ cannot bring the process arbitrarily fast toward the origin. Specifically, we need for every $t \geq s$ and $y$ that
    \[
        \lim_{|x_s| \to \infty} |x_{t}| = \infty.
    \]
    Since $v_t$ is Lipschitz, there exists constants $c_0, c_1$, such that, for all $t$,
    \[
        \frac{d}{\dt}|x_t| \geq - |v_t(x_t)| \geq -(c_0 + c_1 |x_t|),
    \]
    whereby Gr\"onwall's inequality yields the estimate
    \[
        |x_t| \geq |x_s|e^{-c_1 (t-s)} + \frac{c_0}{c_1}(e^{-c_1(t-s)} - 1) \to \infty  \quad \textnormal{ as } \quad |x_s| \to \infty.
    \]
    Hence the space-time process $\cZ$ is Feller and the space-time generator $\fA = \partial_t + \modgen_t$ permits $\cD = C^\infty_c(\cM)$ as a core \cite[Theorem~5.3]{holderrieth2021cores}. Let $\mu_0(\dx)\psi(\dy)$ be the initial law of $(X_0,Y_0)$, which corresponds to $\cZ_0 \sim \delta_0(\dt)\mu_0(\dx)\psi(\dy)$. Since $\cZ$ evolves purely deterministically in the time coordinate, it follows that then $\cZ_t \sim \delta_t(\dt)\rho_t(\dx,\dy)$ for $t \leq T$, and we define $\textnormal{Law}(X_t,Y_t)=\rho_t$.
    For any $f\in\cD$ we find now by Dynkin's formula \cite[Theorem~31.3]{davis1993markov} that, for any $0 \leq s \leq t \leq T$,
    \begin{align*}
        \EE[f(\cZ_t)] - \EE[f(\cZ_0)] &= \EE\left[\int_0^t \fA f(\cZ_s)\dif s\right] \iff \\
        \int f(t,\cdot) \dif\rho_t - \int f(0,\cdot) \dif\rho_0 
        &= \int_0^t \left(\int_{\Rdim \times S^{\ddim -1}}\frac{\partial f}{\partial s} + \modgen_s f \dif\rho_s\right) \dif s \\
        &= \int_0^t \left( \int_{\Rdim\times S^{\ddim -1}} \frac{\partial f}{\partial s} + \langle v_s, \nabla_x f\rangle \dif\rho_s \right)\dif s. \\
    \end{align*}
    The above implies that the law of $X_t$ satisfies the continuity equation \( \nabla \cdot(v_t\mu_t) + \partial_t \mu_t = 0 \). Additionally, the PDMP preserves the velocity distribution for any $t < T$ and the statement follows.
\end{proof}

\subsection{Velocity field and diffusion on \texorpdfstring{$\RR$}{R}}\label{sec:one_dim_case}

In this section, we briefly consider the special case of a one-dimensional potential $\obj$. In this case, we can both show absolute continuity of the curve, and the existence of a diffusion process controlled by the velocity field, under weaker conditions. More specifically, we assume that $\obj$ in this setting satisfies \cref{ass:inf_obj_exists,ass:quadratic_tail}, but we replace \cref{ass:quadratic_gradient} with the following weaker condition:
\begin{assumption}\label{ass:1d_gradient_condition}
    There exists an $R$ and $\alpha > 0$ such that 
    \[
        |x| > R \implies x\left( \beta_0 \obj'(x) - \frac{\obj'(x)}{\obj(x)}\right) \geq \alpha |x|,
    \]
    where $\beta_0$ is the lowest inverse temperature along the curve.
\end{assumption}
For $\obj\colon \RR \to \RR$ the continuity equation \eqref{eq:continuity_equation} reads
\[
    - \frac{\dif }{\dx} \left( v_t \pi_t \right) = \partial_t \pi_t,
\]
which is readily seen to be solved by
\begin{equation}\label{eq:1d_v}
    v_t(x) = - \frac{1}{\pi_t(x)} \int_{-\infty}^x \partial_t \pi_t(y) \dy = \frac{\beta'(t)}{\pi_t(x)} \int_{-\infty}^x \left( \obj (y) - \int\obj\dif\mu_t\right) \pi_t(y) \dy,
\end{equation}
and the main result in this section amounts to verifying that this $v_t$ is sufficiently regular.

\begin{restatable}{theorem}{thmonedimac}\label{thm:1d_ac}
    Assume that $\obj\colon \RR \to \RR$ satisfies Assumption \ref{ass:inf_obj_exists}, \ref{ass:quadratic_tail} and \ref{ass:1d_gradient_condition}. Let $\beta\colon [0,T] \to [1,\infty)$ be a non-decreasing smooth cooling schedule. Then, for any $1 \leq p < \infty$, $v_t$ defined as in \eqref{eq:1d_v} satisfies $\|v_t\|_{L^p(\mu_t)} \in L^1([0,T])$
    and the curve
    \begin{align*}
        [0,T] \ni t \mapsto \mu_t \in \sP_p(\RR)        
    \end{align*}
    is absolutely continuous.
\end{restatable}

We give an elementary proof of the preceding theorem in \cref{app:proofs_1d}. In this, we note that $\int_{-\infty}^x \left( \obj (y) - \int\obj\dif\mu_t\right) \pi_t(y) \dy \to 0$ when $|x| \to \infty$, and hence the assumptions on $\obj$ forces $v_t$ to point towards the origin in the tails. To conclude this section, we note that this control on the sign of $v_t$ for large $|x|$ additionally allows us to drop the linear growth assumption used in \cref{sec:fss_diffusive}, yielding a stronger result as follows.

\begin{proposition}
    Assume that $\obj$ and $\beta$ are as in Theorem \ref{thm:1d_ac} and let $v_t$ be as in \eqref{eq:1d_v}. Then the stochastic differential equation on $\RR$ given by
    \[
        \dif X_t = \left( v_t(X_t) - \nabla \obj(X_t) \right) \dt + \sqrt{\frac{2}{\beta(t)}} \dif B_t
    \]
    has a weak solution, and with $X_0 \sim \mu_0$ we have that \( \textnormal{Law}(X_t) = \mu_t \) for almost every $t \in [0,T]$.
\end{proposition}

\begin{proof}[Sketch of proof]
    The proof follows analogously to those for \cref{prop:weak_sol_SDE_plus_v} and \cref{prop:gibbs_curve_sol_fp_vsde}. 
    However, in the application of \cite[{Theorem~10.2.2}]{stroock1997multidimensional} we instead use Assumption \ref{ass:1d_gradient_condition} and the sign of $v_t(x)$ for large $|x|$ to conclude that the conditions are satisfied.
\end{proof}

\section{Tractable approximations to controlled annealing processes}\label{sec:num_approx}%
So far we have shown the existence of a velocity field which allows for arbitrarily fast cooling. In practice, computing the velocity field is difficult, as it is equivalent to solving the PDE \eqref{eq:gibbs_continuity_equation} at every time $t$.
Nonetheless, we will now introduce a series of approximations which in the end yields a tractable numerical scheme inspired by the idealized processes. In particular, we approximate the flow of the measures along the Gibbs curve using a finite number of interacting particles.

To this end, we first note that for any $t$, the limit in \cref{prop:v_as_monge_limit} yields that for any $\epsilon > 0$ there exists some $h = h(\epsilon) > 0$ such that
\[
   \left\|h^{-1}(\Ttth - \textrm{Id}) - v_t \right\|_{L^2(\mu_t)} < \epsilon.
\]
In other words, at any $t$, we could approximate $v_t$ arbitrarily well by the finite difference quotient $h^{-1}(\Ttth - \Id)$. Computing the Monge map $\Ttth$ is however still intractable. With access to independent random variables $X^1,...,X^n \sim \mu_t$, $Y^1,...,Y^n \sim \mu_{t+h}$ and denoting by $\mu^n_t, \mu^n_{t+h}$ the corresponding empirical probability measures, one estimator for $\Ttth$ would be the solution to the corresponding assignment problem. 
However, assuming that we are able to sample from $\mu_{t+h}$ means that we would
already have solved the problem of following the Gibbs curve.
Nevertheless, with access only to samples from $\mu_t$, we can make $\mu^n_t$ resemble $\mu^n_{t+h}$ via importance sampling--like re-weighting, and then estimate the transport map by solving the transport problem between these discrete measures. In the next section, we establish that this procedure converges as $n \to \infty$, and in \cref{sec:algs} we utilize this to design our numerical methods.

\subsection{Self-normalized importance sampling estimation of optimal transport maps}

Here, we briefly consider finite approximations to the minimizer $T$ of the Monge problem
\begin{equation}\label{eq:monge_w2}
    \inf_{T\colon T_\#\mu = \nu} \int |x-T(x)|^2 \difempty\mu(\dx),
\end{equation}
when $\mu, \nu \in \Psp_2(\Rdim)$ permit unnormalized Lebesgue densities $p,q$, and $\nu \ll \mu$. Specifically, we are interested in estimators to $T$ from independent random variables $X^1,X^2,...$ distributed according to $\mu$. 
By the law of large numbers, the sequence of empirical distributions $\mu^n = n^{-1}\sum_{i=1}^n \delta_{X^i}$ converges weakly to $\mu$. In fact,
with normalized weights $w_i^n \propto q(X^i) / p(X^i)$, we additionally find that the empirical importance sampling measure $\nu^n = \sum_{i=1}^n w_i^n \delta_{X^i}$ converges weakly to $\nu$, a standard result on importance sampling \cite[Section~3.3]{robert2005monte}.

As the Monge problem \eqref{eq:monge_w2} is not in general solvable for measures supported on point clouds, we cannot a priori ask whether the solutions to the finite Monge problems from $\mu^n$ to $\nu^n$ converge. To remedy this, we instead consider conditional projections $T^n$ of the optimal transport plans $\gamma^n \in \Gamma(\mu^n,\nu^n)$, defined as 
\[
    T^n(x) := \EE_{(X,Y) \sim \gamma^n}[Y \mid X = x] = \int y \difempty\gamma^n(x,\dy).
\]
In general, $T^n_\# \mu^n \ne \nu^n$, i.e., $T^n$ does not define an admissible map. Further, as $T^n$ is not uniquely defined $\mu$ almost everywhere, we cannot hope for strong convergence in $L^2(\mu)$. It does however hold that $T^n \to T$ in a suitable weak sense and a suitable norm sense.

\begin{proposition}\label{prop:importance_transport}
    Let $\mu,\nu \in \Psp_2(\Rdim)$ be two probability measures permitting unnormalized Lebesgue densities $p$ and $q$, such that $p(x)=0 \implies q(x) = 0$. Let $T$ be a Monge map from $\mu$ to $\nu$ solving \eqref{eq:monge_w2}.
    Let $X^1,X^2,...$ be independent random variables such that $X^i \sim \mu(\dx)$ for all $i$ and define the empirical measures
    \[
        \mu^n := \frac{1}{n}\sum_{i=1}^n\delta_{X^i}, \quad 
        \nu^n := \sum_{i=1}^n w^n_i \delta_{X^i} , \quad w_i := \frac{q(X^i)}{p(X^i)}, \quad w_i^n := \frac{w_i}{\sum_{j=1}^n w_j}.
    \]
    Let $\gamma^n$ be an optimal $\wass_2$-transport plan of $\mu^n,\nu^n$,
    \[
        \int_{\Rdim\times\Rdim} |x-y|^2 \difempty\gamma^n(\dx,\dy) = \inf_{\gamma \in \Gamma(\mu^n,\nu^n)} \int_{\Rdim \times \Rdim}|x-y|^2 \difempty\gamma(\dx,\dy),
    \]
    and let $T^n(x) = \int y\difempty\gamma^n(x,\dy)$ be the barycentric projection of $\gamma^n$. Then as $n\to\infty$, almost surely
    \[
        T^n_\#\mu^n \to T_\#\mu =\nu \quad \text{weakly} \quad 
    \]
    and
    \[
        \limsup_n  \|T^n\|_{L^2(\mu_n)} \leq  \|T\|_{L^2(\mu)}.
    \]
\end{proposition}

\begin{proof}
    By the strong law of large numbers, both $\mu^n \to \mu$ and $\nu^n \to \nu$ weakly.
    As $\mu, \nu \in \Psp_2(\Rdim)$, then, almost surely, $\mu^n, \nu^n \in \Psp_2(\Rdim)$. Hence, by \cite[Proposition~7.1.3]{ambrosio2008gradientflows}, the sequence of optimal transport plans $(\gamma^n) \subset \Psp(\Rdim \times \Rdim)$ has weak limit points, and all such limit points $\gamma$ satisfy 
    \[
        \gamma \in \argmin_{\gamma' \in \Gamma(\mu,\nu)} \int_{\Rdim\times\Rdim}|x-y|^2\difempty\gamma'(\dx,\dy).
    \]
    As $\mu$ has a Lebesgue density, $\gamma$ is unique. Therefore, the entire sequence $\gamma^n$ in fact converges weakly to $\gamma$. 

    To show that $T^n$ converges to $T$ in the prescribed sense, we seek to apply \cite[Theorem~5.4.4]{ambrosio2008gradientflows}. To this end, we estimate
    \begin{align*}
        \| T^n\|^2_{L^2(\mu^n)} &= 
         \frac{1}{n} \sum_i \left|\int y\difempty\gamma^n(x_i,\dy) \right|^2 \\
        &\leq \frac{1}{n} \sum_i \int |y|^2\difempty\gamma^n(x_i,\dy)\\ 
        &= \int |y|^2\difempty\gamma^n (\dx,\dy) \\
        &= \int |y|^2 \difempty \nu^n(\dy) = \sum_{i=1}^n |x_i|^2 w_i,
    \end{align*}
    using Jensen's inequality, and that $\gamma^n$ is a coupling between $\mu^n$ and $\nu^n$. As $\nu \in \sP_2(\Rdim),$ the strong law of large numbers yields that
    \[
        \limsup_n \left|\sum_{i=1}^n |x_i|^2 w_i - \int |y|^2 \difempty\nu(\dy) \right| =0 \quad \text{almost surely.}
    \]
    As $\int |y|^2 \difempty\nu(\dy) = \int |T(x)|^2 \difempty\mu(\dx) =   \|T\|^2_{L^2(\mu)}$, we additionally find
    \[
        \limsup_n \|T^n\|^2_{L^2(\mu^n)} \leq \|T\|^2_{L^2(\mu)}.
    \]
    With $f\in C^\infty_c(\Rdim)$
    \begin{align*}
        \lim_n \int_\Rdim f(x) T^n(x)\difempty \mu^n(\dx) &= 
         \lim_n \int_\Rdim f(x) \left(\int_\Rdim y \difempty\gamma^n(x, \dy)\right) \difempty\mu^n(\dx)\\
         &= \lim_n \int_{\Rdim\times\Rdim} f(x) y \difempty\gamma^n(\dx,\dy) \\
         &= \int_{\Rdim\times\Rdim} f(x)y \difempty\gamma(\dx,\dy),
    \end{align*}
    since $\gamma^n \to \gamma$ weakly. An application of \cite[Theorem~5.4.4]{ambrosio2008gradientflows} now readily yields that, almost surely,
    \[
        T^n_\#\mu^n \to T_\#\mu = \nu \quad %
    \]
    as desired.
\end{proof}

\subsection{Simulating approximately controlled annealing processes}\label{sec:algs}

\newcommand{\bx}{\mathbf{x}}
\newcommand{\bv}{\mathbf{v}}
\newcommand{\bX}{\mathbf{X}}
\newcommand{\bY}{\mathbf{Y}}
\newcommand{\bV}{\mathbf{V}}
\newcommand{\bZ}{\mathbf{Z}}
\newcommand{\bT}{\mathbf{T}}

We now describe tractable numerical approximation schemes of the processes in \cref{sec:controlled_superposition}, relying on a finite difference approximation in time and a particle approximation in space, with parameters $h,n$ respectively. The final algorithms consist of simulating a population $\bX_t$ of $n$ particles,
\[
    \bX_t = (X^1_t,...,X^n_t) \in \Rdim \times ... \times \Rdim,
\]
and $n$ piecewise constant velocity components
\[
    \bV_t = (V^1_t,...,V^n_t) \in \Rdim \times ... \times \Rdim.
\]
Each particle $X^i_t$ evolves quasi-independently and is driven by independent stochastic dynamics, and the velocity $V^i_t$, which is an approximation of the true velocity field. It is the velocities $\bV_t$ that introduce interaction between the particles.
At predetermined times, we compute $\bV_t$ as follows: given that particle $X^i_t = x_i$ at time $t$ for all $i$, we define the reweighted empirical measure $\hat\mu_{t+h}^n$ as in \cref{prop:importance_transport} by
\begin{equation}\label{eq:reweighting}
    \hat\mu^n_{t+h} := \sum_{i=1}^n w_i\delta_{x_i}, \quad w_i \propto \exp\left(-\obj(x_i)(\beta(t+h)-\beta(t))\right), \quad \sum_{i=1}^nw_i = 1.
\end{equation}
With $C_{ij} := |x_i - x_j|^2$ and identifying that any $\gamma \in \Gamma(\hat \mu_t^n, \hat \mu_{t+h}^n)$ can be expressed as $\gamma= \sum_{ij} G_{ij}\delta_{x_i}(\dx)\delta_{x_j}(\dy)$, the transport problem between the two empirical measures is the finite-dimensional linear program
\begin{alignat*}{3}
    G^* := & \argmin_{G\geq 0}    & \quad &  & \langle C, G    & \rangle_{\textrm{Fr}}   \label{eq:transport_lp}\tag{LP} \\
           & \,\,\,\,\,\,\,\text{s.t.} &       &  & G\mathbf{1} =   & \, \mathbf{1}                             \\
           &                   &       &  & G^T\mathbf{1} = & \, n\mathbf{w},
\end{alignat*}
where $\mathbf{w} = (w_1,...,w_n)^T$, $\langle C, G \rangle_{\textrm{Fr}} := \sum_{ij} C_{ij}G_{ij}$ denotes the Frobenius inner product, and $\mathbf{1}$ denotes a vector of ones.
By conditional projection of the solution $G^*$ to this linear program, we find the transport map estimator
\[
    T(x_i) = \EE_{(X,Y)\sim\gamma^* }[Y \mid X = x_i] = \left(\sum_j G^*_{ij}x_j\right),
\]
and set the components of $\bV_t$ to
\begin{equation}\label{eq:alg_v_from_lp}
    V_t^i = \frac{\left(\sum_jG^*_{ij}x_j\right)-x_i}{h}.
\end{equation}
Barring any other dynamics, this velocity $V^i_t$ would then move $X^i_t$ to $T(X^i_t)$ in a straight line in $h$ time.

\paragraph{Diffusive dynamics} To integrate the diffusive dynamics, we use an Euler-Maruyama discretization with a fixed step length $\Delta t$, and we pick a velocity update interval $k\Delta t$ for some positive integer $k$. The resulting numerical scheme, described in \cref{alg:langevin_transport}, is easily parallelisable over all steps except solving \eqref{eq:transport_lp}.

\begin{algorithm}
    \caption{Controlled Langevin diffusion for simulated annealing.}\label{alg:langevin_transport}
    Given an objective function $\obj$, a cooling schedule $\beta$, initial states $\bX_0 =\left(X^i_0\right)_{i=1}^n$, a time discretization $\Delta t$, a velocity update interval $h = k\Delta t$, and an end time $T$. Set $t = 0$.
    \begin{enumerate}
        \item \emph{Estimate the velocity field}:
        \begin{enumerate}
        \item Compute the weights $\tilde w_i = \exp(-(\beta(t+h)-\beta(t))\obj(X_t^i))$, the normalized weights $w_i = \tilde w_i / \sum_j \tilde w_j$, and the cost matrix $C_{ij} = |X_t^i - X_t^j|^2$.
        \item Solve the discrete optimal transport problem $G^* = \argmin_{G_{ij}\geq 0} \{ \langle G, C\rangle_{\textrm{Fr}}\colon G\mathbf{1} = \mathbf{1}, G^T\mathbf{1} = n\mathbf{w} \}.$
        \item Set
        $V^i_t = h^{-1}\left(\left( \sum_j G^*_{ij}X^j_t\right) - X_t^i\right)$ for  $i = 1,...,n.$
    \end{enumerate}
        \item \emph{Integrate the SDE}: For every particle $i$ and $\ell=0,...,k-1$ sample independently $\zeta^i_\ell \sim \cN(0, I_d)$ and set
        \[
            X_{t + (\ell+1)\Delta t}^i = X^i_{t + \ell\Delta t} + \Delta t \left( V_t^i - \nabla \obj (X_{t + \ell\Delta t}^i)\right) + \sqrt{2 \beta^{-1}(t + \ell\Delta t)\Delta t} \zeta^i_\ell.
        \]
        Set $t = t + k\Delta t$ and repeat from Step 1 if $t < T$.
    \end{enumerate}
\end{algorithm}

\paragraph{Synchronous PDMP} For $n$ particles, we can formalize the algorithm as a single PDMP on $\RR^{3 \times \ddim \times n}$ with scheduled events. For the process at time $t$, we simulate independently a random bounce time $\Delta^i$ for each particle and then advance the ensemble forward in time $\Delta =\min_i(\Delta^i)$ if $t+\Delta$ is before the next scheduled velocity estimation. If the next velocity computation is scheduled before $t+\Delta$, then the particles are instead propagated up to this time and the velocity is recomputed. To this end, let $\bZ_t = (\bX_t,\bY_t,\bV_t)$ and define
\begin{align*}
    \xi(s,t,\bZ_t) &= (\bX_t + s(\bY_t + \bV_t), \bY_t, \bV_t), \\
    \lambda^i_b(t, \bZ_t)  &= \max(0, \beta(t)\langle \nabla\obj(X^i_t),  Y^i_t \rangle), \\
    \cQ^i_b(\bZ_t, \dif z) &= \delta_{(\bX_t, (Y^1_t,...,Y'_t, Y^{i+1}_t,...,Y^n_t), \bV_t)}(\dif z), \quad Y'_t = R(X^i_t)Y^i_t.
\end{align*}
Formally, we find a joint rate $\lambda$ and a mixture kernel $\cQ$
\begin{align*}
    \lambda(t, \bZ_t) &= \lambda_r + \sum_i \lambda^i_b(t, \bZ_t), \\
    \cQ(t, \bZ_t, \dif z) &= \frac{\lambda_r}{\lambda(t, \bZ_t)} \cQ_r(\dif z) + \sum_i \frac{\lambda^i_b(t, \bZ_t)}{\lambda(t,\bZ_t)}\cQ^i_b(\bZ_t, \dif z),
\end{align*}
but in practice we understand this as sampling from the event that arrived first. The refreshment kernel $\cQ_r$ resamples $\bY_t$ for one or more particles. With an update schedule $((t_k, \cQ_k))_k$, $t_k = hk$ and $\cQ_k$ a deterministic kernel supported on the velocity estimate from the barycentric projection of the solution of \eqref{eq:transport_lp}, this scheme is implementable using the recursive construction in  \cref{alg:recursive_pdmp_m}.

\begin{algorithm}
    \caption{Recursive construction of a PDMP with scheduled events.}
    Let $Z_t$ be the state of a PDMP at time $t$, governed by an event rate $\lambda$, a flow $\xi$, a transition kernel $\cQ$, and an event schedule $((t_k, \cQ_k))_k$.
    \begin{enumerate}
        \item Let $t^* = t_{k^*} = \min \{t_k \colon t_k >t\}$ be the arrival of the next scheduled event.
        \item Simulate independently $\Delta$ as the first arrival time of an inhomogeneous Poisson process with rate $s \mapsto \lambda(t+s,\xi(s, t, Z_t))$, that is, a $[0,\infty)$-valued random variable with distribution function
        \[
            F(\tau) = \exp\left( -\int_0^\tau \lambda(t+s,\xi(s, t, Z_t)) \dif s \right), \quad \tau \geq 0.
        \]
        \item Let $t' = \min(t^*, t +\Delta)$ and set for $s \in [0,t'-t)$
        \[
            Z_{t+s} = \xi(s, t, Z_t).
        \]
        \item\begin{enumerate}
            \item \emph{If $t' = t + \Delta$}: Draw
            \[
                Z_{t'} \sim \cQ(\xi(t'-t,t,Z_t), \dif z).
            \]
            \item \emph{Else}, $t' = t^*$: Draw
            \[
                Z_{t'} \sim \cQ_{k^*}(\xi(t'-t,t,Z_t), \dif z).
            \]
        \end{enumerate}
        \item Repeat with $Z_{t'}$.
    \end{enumerate}
    \label{alg:recursive_pdmp_m}
\end{algorithm}

\paragraph{Asynchronous PDMP} In contrast to the discretization of the diffusive dynamics, where it was simple to propagate each particle forward in time in parallel, synchronization of the BPS dynamics comes at a considerable cost. Indeed, as the random dynamics affecting the particles in between velocity updates are completely independent, it is possible to let the particles evolve asynchronously in time until all of them have reached the boundary of the next scheduled velocity update. Similarly to the construction used in the proof of \cref{prop:mbps_works}, we introduce the process
\[
    \mathbf{\cal{Z}}_t = (\mathbf{T}_t, \bX_t, \bY_t, \bV_t) \in (\RR_+ \times \Rdim \times S^{\ddim -1} \times \Rdim)^n.
\]
When the time coordinate $T^i_t$ of a particle $i$ lies in $[\tau_k,\tau_{k+1})$, it follows the deterministic dynamics
\[
    \xi^i(s, t, (T^i_t, X^i_t,Y^i_t,V^i_t)) = (T^i_t + s, X^i_t + s(Y^i_t + V^i_t), Y^i_t, V^i_t)
\]
and upon reaching the boundary $T^i_t = \tau_{k+1}$ the value of $T_t^i$ stops increasing until all particles have reached the boundary. When this happens, the velocity $\bV_t$ is recomputed and the particles evolve forward again on $[\tau_{k+1}, \tau_{k+2})$.

\subsection{Connections to other methods}

We end this section with a few remarks about the connection between the algorithms we suggest and other algorithms in the literature.

\begin{remark}[Consensus-based optimization]\label{rem:cbo}
    It is natural to compare our numerical method to other ensemble-based optimization methods. In consensus-based optimization \cite{carrillo2018analytical, tretyakov2023consensus}, each particle $X^i_t$ follows an SDE with a drift term $b^i(\bX_t)$ which steers the particle to the current \emph{consensus point} $c_\alpha(\bX_t)$ of the ensemble. Here \( c_\alpha(\bX_t) = \int x \hat\rho^n_t(\dx) \), where $\hat\rho^n_t$ is the probability measure defined by \( \hat\rho^n_t \propto \sum_{i=1}^n \exp(-\alpha\obj(X^i_t))\delta_{X^i_t}  \) and $b^i(\bX_t) = \lambda (c_\alpha(\bX_t) - X^i_t)$ for some chosen $\lambda > 0$. Here, the parameter $\alpha$ tunes the concentration towards the current best estimate. In contrast, our particles follow more inhomogeneous dynamics, as they are driven not towards a single global consensus point, but instead towards some individually assigned point such that the drift of the entire ensemble achieves a consensus effect. 
\end{remark}

\begin{remark}[Birth-Death processes and Sequential Monte Carlo Simulated Annealing]\label{rem:smcsa}
While we have focused on solving the Fokker--Planck equation \eqref{eq:fp_question_intro} using velocity fields, it is also possible to instead consider other types of stochastic processes. We restrict now our attention to processes taking values in a compact subset $\cX\subset\Rdim$. Assume that the inverse temperature $\beta$ is non-decreasing, $\beta^\prime \ge 0$, and assume that $U$ is continuous, in which case both $\sup_{x\in\cX}\obj(x)$ and $\inf_{x\in\cX}\obj(x)$ are bounded. Let $\hat U(x) = U(x) - \inf_{x\in\cX}U(x) \ge 0$.
The operator
\[
\mathcal Q_t f(x) = |\beta'(t)\hat U(x)| \int \left(f(y) - f(x)\right)\pi_t(y) \dy 
\]
is the infinitesimal generator of a pure jump process,
\cite[Proposition~17.2]{kallenberg2021foundations}. 
This process jumps with the position-dependent rate $|\beta'(t)\hat U(x)|= \beta'(t)\hat\obj(x)$ to a random location independently sampled from $\mu_t$.
This process also formally solves the Fokker--Planck equation \(\partial_t\mu_t = \cQ_t^*\mu_t\) in the sense that
\begin{align*} 
\int \mathcal Q_t f(x) \pi_t(x) \dx &=
    \int \left(|\beta'(t)\hat U(x)|\int \left(f(y) - f(x)\right)\pi_t(y) \dy \right) \pi_t(x) \dx \\
&= \beta'(t)\left(
        \int f(y) \pi_t(y) \dy \int\hat U(x)\pi_t(x) \dx
        - \int \hat U(x)f(x)  \pi_t(x) \dx\right) \\
&=  - \beta'(t)\,\covt \left(f, \hat U \right)
= -\beta'(t)\,\covt \left(f,  U \right)
= \int 
    f(x)\partial_t \pi_t(x) \dx,
\end{align*}
see \eqref{eq:dpidt}.
A natural way to numerically implement this process is by using a particle swarm approximation to $\mu_t$, removing particles with the position-dependent rate $|\beta'(t)\hat U(x)|$, and recreating them at a new location, drawn independently from the locations of the pool of remaining particles.
This can be viewed as a continuous-time version of Sequential Monte Carlo Simulated Annealing \cite{zhou2013sequential}, a discrete time simulated annealing algorithm where particles are successively resampled over time according to suitably chosen weights. 

\end{remark}

\begin{remark}[Predictor--corrector]
     At the end of \cref{sec:related_work}, we note that our view of simulated annealing is reminiscent of an interior point method. In this context, it is also interesting to note that our proposed algorithms could arguably be seen as predictor--corrector-type methods \cite{mehrotra1992implementation, mizuno1993adaptive}. More precisely, estimating and using the velocity field $v_t$ can be interpreted as a predictor-type step, which estimates how to move the particles to follow along the curve. Analogously, the superposition with stochastic $\mu_t$-invariant and ergodic dynamics can be interpreted as a corrector step, which moves the particles closer to a configuration that represents the current point on the curve.
\end{remark}

\section{Numerical experiments} \label{sec:numerical_experiments}

\def\placeholderplots{0}

We perform a number of numerical experiments to illustrate the behaviour and to examine the practical applicability of the proposed method. Before this, we note that the velocity estimates and the independent stochastic dynamics can run on different timescales. For any potential $\obj$, we fix the time interval to $[0,1]$ and a cooling schedule $\beta$. After choosing a time discretization $\Delta t$ for the Langevin dynamics, we can freely choose any parameter $\lambda > 0$ and instead integrate
\[
    \dif X_t = -\lambda\nabla\obj(X_t)\dt + \lambda\sqrt{2{\beta(t)}^{-1}} \dif B_t,
\]
which, in practice, amounts to finding a stable gradient descent step length $\lambda\Delta t$. In the PDMP version, this corresponds to choosing a scaling $\lambda>0$ for the BPS velocity component $Y_t \in \lambda S^{\ddim -1}$. For the simulations we have performed, we present the normalized time interval $[0,1]$ and give the corresponding choice of $\lambda\Delta t$ or $\lambda$, respectively. For brevity, we abbreviate the PDMP-based piecewise deterministic simulated annealing method as PDSA.
In the sequel, we have implemented the synchronous version of the controlled PDSA, which lends itself well to the qualitative studies we perform.
We provide a reference implementation and code for reproducing the results presented in this section at \url{https://github.com/vincentmolin/controlled_annealing}.

\subsection{A double-well potential on \texorpdfstring{$\RR$}{R}}\label{sec:exp_doublewell}

\begin{figure}
    \centering
    \begin{subfigure}[t]{0.4\textwidth} \vskip 0pt
    \includegraphics{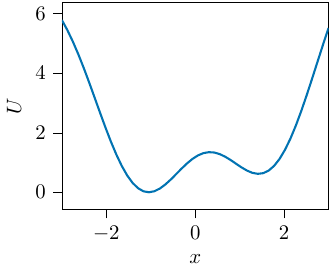}
    \end{subfigure} \hspace{0.05\textwidth} %
    \begin{subfigure}[t]{0.3\textwidth} \vskip 0pt
    \includegraphics{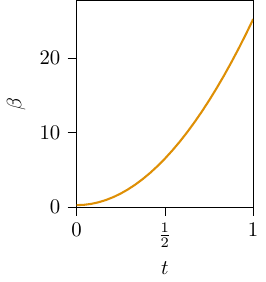}
    \end{subfigure}
    \caption{The double-well potential $\obj$, (left) and the cooling schedule $\beta$ of \cref{sec:exp_doublewell}.}
    \label{fig:doublewell}
\end{figure}

We first consider the double-well potential $\obj\colon\RR\to\RR$ given by
\begin{equation}\label{eq:doublewell}
    \obj(x) = \frac{1}{2}x^2 + \cos\left(2x-\frac{1}{2}\right) + C,
\end{equation}
where $C$ is chosen such that $\min \obj = 0$. This potential has a unique global minimum close to $x=-1$, and a local suboptimal minimum in the vicinity of $x=1.5$. Our Gibbs curve is then defined by choosing the cooling schedule given by $\beta\colon t\mapsto \frac{1}{4} + 25{t^2}$ for $t\in[0,1]$. In particular, we initialize each particle by sampling from $\mu_0$. Here, as $\beta$ increases quickly, the independent stochastic dynamics of particles cause them to stick in the suboptimal well. Simulating the particles instead in groups according to the controlled algorithms allows more flow of probability mass to the global optimum. \cref{fig:doublewell_langevin_hists} illustrates the phenomenon for the diffusion-based method, as \cref{fig:doublewell_pdsa_hists} does for the PDMP version.

\begin{figure}
    \centering
    \begin{subfigure}[b]{0.43\textwidth}
        \includegraphics{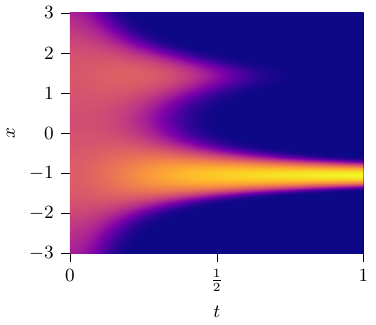}
        \caption{Ground truth.}
    \end{subfigure} %
    \begin{subfigure}[b]{0.43\textwidth}
        \includegraphics{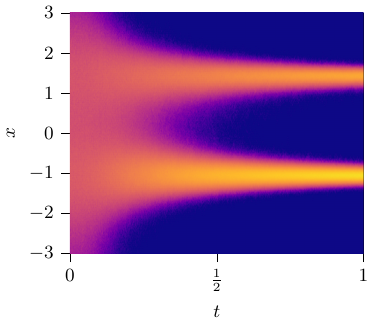}
        \caption{Langevin, independent particles.}
    \end{subfigure} %
    \begin{subfigure}[b]{0.43\textwidth}
        \includegraphics{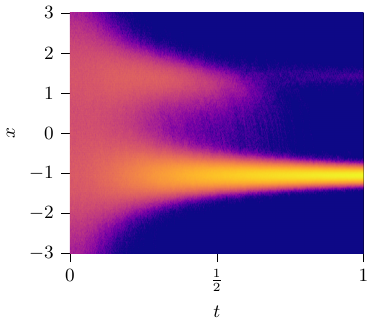}
        \caption{Controlled Langevin, 5 particles.}
    \end{subfigure} %
    \begin{subfigure}[b]{0.43\textwidth}
        \includegraphics{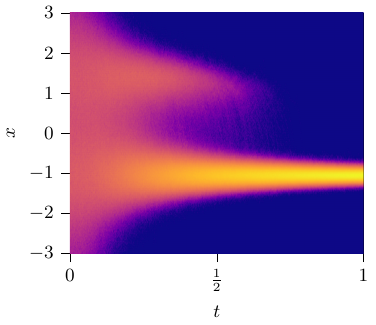}
        \caption{Controlled Langevin, 10 particles.}
    \end{subfigure} %
    \caption{Heat maps of empirical probability densities for the Gibbs measures corresponding to the double-well potential $\obj$ in \eqref{eq:doublewell} and the quadratic cooling schedule $\beta(t) = \frac{1}{4} + 25t^2$, for $0 \leq t  \leq 1$. The Langevin versions are simulated with 1000 steps of length $\lambda\Delta t = 25\cdot 10^{-3}$. For the controlled versions we set $h=2\cdot 10^{-2}$, that is, we compute a new velocity estimate 50 times.
    From the upper right panel, we see that the fact that $\beta$ increases quickly causes particles obeying the independent Langevin dynamics to stick in the suboptimal well. In contrast, from the two panels at the bottom row we see that the interactions in the proposed method allows particles to escape the local minima.
    Histograms are binned averages over 10,000 particle trajectories, coloured with an inverse hyperbolic sine colour scale.}\label{fig:doublewell_langevin_hists}
 \end{figure}
 
\begin{figure}
    \centering
    \begin{subfigure}[b]{0.43\textwidth}
        \includegraphics{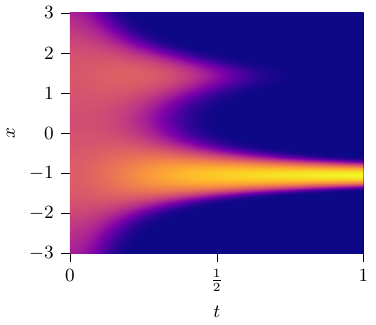}
        \caption{Ground truth.}
    \end{subfigure} %
 \begin{subfigure}[b]{0.43\textwidth}
        \includegraphics{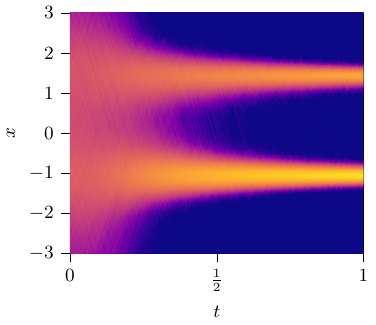}
        \caption{PDSA, independent particles.}
    \end{subfigure} %
    \begin{subfigure}[b]{0.43\textwidth}
        \includegraphics{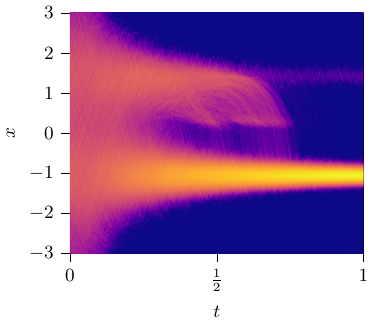}
        \caption{Controlled PDSA, 5 particles.}
    \end{subfigure} %
    \begin{subfigure}[b]{0.43\textwidth}
        \includegraphics{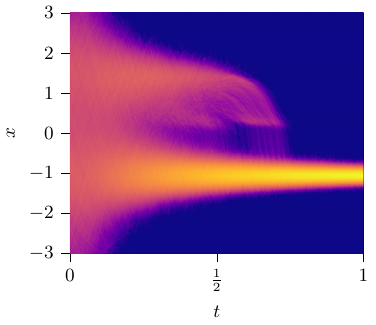}
        \caption{Controlled PDSA, 10 particles.}
    \end{subfigure} %
    \caption{Heat maps of empirical probability densities for the Gibbs measures corresponding to the double-well potential $\obj$ in \eqref{eq:doublewell} and the quadratic cooling schedule $\beta(t) = \frac{1}{4} + 25t^2$, $0\leq t \leq 1$. The PDSA versions are simulated with a speed scale $\lambda = 25$, and for the controlled versions we set $h=2\cdot 10^{-2}$, that is, we compute a velocity estimates 50 times. 
        As in the Langevin case in \cref{fig:doublewell_langevin_hists}, the independent PDSA dynamics of particles stick in the suboptimal well. In contrast, the interactions in the proposed method allows particles to escape the local minima.
        Histograms are binned averages over 10,000 particle trajectories, coloured with an inverse hyperbolic sine scale.}\label{fig:doublewell_pdsa_hists}
\end{figure}

The most dramatic flow of mass happens close to the centre of the time interval $[0,1]$. In \cref{fig:doublewell_w2_v} we quantify this by computing the ground truth metric derivative, which has a fairly sharp peak in the middle. By also estimating the $\wass_2$-distance of the average marginal law of the particles and the ground truth, we see that this is the region where the laws diverge from the Gibbs curve. As can be seen in the figure, increasing the population of interacting particles counteracts this to some extent. 

\begin{figure}
    \centering
    \begin{subfigure}[t]{0.55\textwidth}%
            \includegraphics{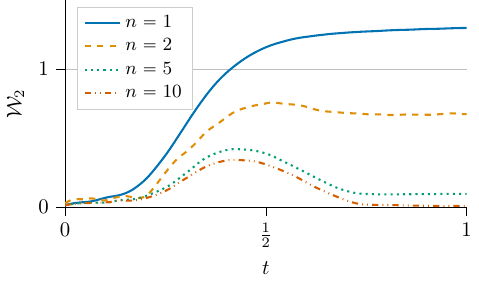}
        \caption{$\wass_2$-distance to the Gibbs curve $\mu_t$.} 
    \end{subfigure} \hspace{0.05\textwidth} %
    \begin{subfigure}[t]{0.35\textwidth}%
        \includegraphics{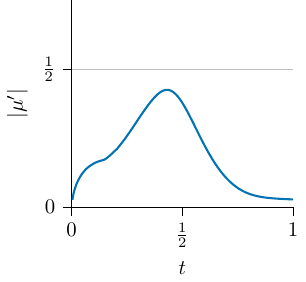}
        \caption{The speed of the Gibbs curve in $\Psp_2$.}
    \end{subfigure}
    \caption{In the left panel, the distance of the time marginals from the Langevin simulations in \cref{fig:doublewell_langevin_hists} is shown. Here $n$ is the population size, and $n=1$ corresponds to independent dynamics, while $n=2,5,10$ denotes controlled versions. 
    The largest deviation for the controlled methods happens around $t=0.4$, at which point the Gibbs curve has the largest rate of change, as quantified by the metric derivative $|\mu'|(t) = \|v_t\|_{L^2(\mu_t)}$, see the right panel.}\label{fig:doublewell_w2_v}
\end{figure}

Finally, we estimate the convergence of the laws to the Gibbs curve as we increase the number of particles and the frequency of the velocity updates,
keeping the other simulation parameters specified in \cref{fig:doublewell_langevin_hists} and \cref{fig:doublewell_pdsa_hists} fixed. The results, shown in \cref{fig:doublewell_convergence}, suggest that the convergence of the controlled PDSA marginals to the Gibbs curve approaches a power law with respect to the number of particles, while we do not observe this scaling in the Langevin setting.

\begin{figure}
    \centering
    \includegraphics{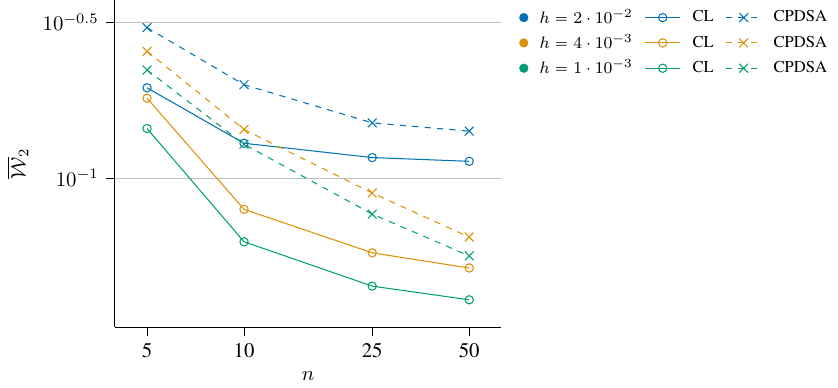}
    \caption{Convergence to the double-well Gibbs curve as the number of particles $n$ increases and the velocity update interval $h$ decreases. Varying the number of particles $n$ (horizontal axis) and the velocity update interval $h$ (colour), we average 1,000 simulations to estimate the time marginals $\rho^{(n,h)}_t$ of the simulations, and compute the time averaged $\wass_2$-distance, $\overline{\wass}_2 = \int_0^1 \wass_2(\rho^{(n,h)}_t,\mu_t)\dt$. Both axes are logarithmically scaled. For smaller $h$, the average distance of the controlled PDSA marginals seems to approach a linear law in this scaling. This suggests that the average distance of the controlled PDSA marginals approaches a power law with respect to the number of particles, while we do not observe this in the Langevin setting.}
    \label{fig:doublewell_convergence}
\end{figure}

\subsection{Optimization problems}

We now evaluate the effect of the velocity estimations by considering two versions of common benchmark problems in optimization on $\Rdim$. Specifically, we set $\ddim = 10$ and consider the Rosenbrock function
\begin{equation}\label{eq:rosenbrock}
    \obj_1(x) := \sum_{i=1}^{\ddim-1} 5 \left( x_{i+1} - x_i^2\right)^2 + (1-x_i)^2,
\end{equation}
and the Rastrigin function given by
\begin{equation}\label{eq:rastrigin}
    \obj_2(x):=|x|^2 + \sum_{i=1}^{\ddim} \left(1 - \cos(2\pi x_i)\right),
\end{equation}
both scaled as in \cite{guilmeau2021simulated}. Moreover, and in contrast to the previous experiment, we now initialize the particles away from the Gibbs curve to model a more realistic scenario. 

For the Rosenbrock function, we use the linear cooling schedule $\beta(t) = 0.1 + 25t$. We draw the initial configurations independently from $\cN(m_i, \frac{1}{20}I)$, centred at $m_1 = (-1.5, 1,1,...,1)$. For the Langevin versions, we use a total of 500 steps of length $\lambda\Delta t=5\cdot \frac 1 {500}$. We run the PDSA versions on a timescale of $\lambda=10$. We simulate controlled versions of both with $5$ particles and 25 velocity updates, and for comparison, we also simulate the independent versions in groups of $5$. In all simulations, we
track the current best objective value $\min_{i=1,...,5} U_1(X^i_t)$ of any particle. In Figure \ref{fig:optim_rosenbrock} we plot the median of this metric across 2,000 simulations, and in \cref{fig:optim_rosenbrock_boxplots} we illustrate the observed variance between these simulations in the form of box plots. For the Langevin versions, we find essentially no gain in performance from the velocity estimates. While there is a stronger effect in the PDSA case, this is mainly due to the poor performance of the independent PDSA algorithm. 

\begin{figure}
    \hspace{0.3in}
    \includegraphics{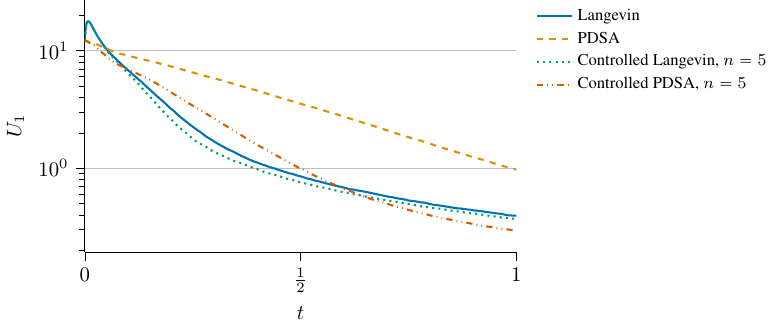}
    \caption{Results for the simulations on the Rosenbrock potential $\obj_1$ in \eqref{eq:rosenbrock}. Out of 2,000 total simulations with $5$ particles, we report the median best of $k=5$ objective value, $\min_{i\leq k} \obj_1(X^i_t)$, simulated independently (Langevin, PDSA) or in groups of $5$.} 
    \label{fig:optim_rosenbrock}
\end{figure}

\begin{figure}
    \hspace{0.3in}
    \includegraphics{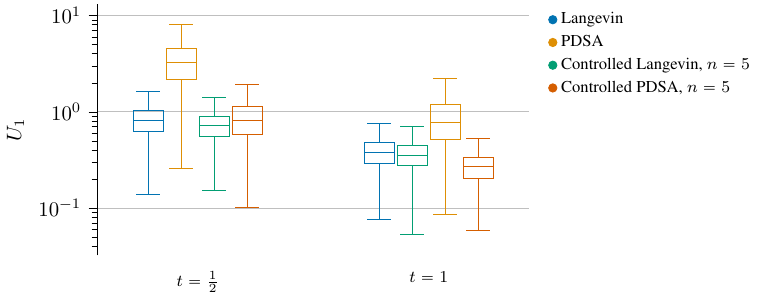}
    \caption{Box plots of the best objective values at $t=\frac{1}{2}$ and $1$, showing the run-to-run variance of the results from the Rosenbrock simulations in \cref{fig:optim_rosenbrock}.} 
    \label{fig:optim_rosenbrock_boxplots}
\end{figure}

For the highly multimodal Rastrigin potential, we use the linear cooling schedule $\beta(t)=0.1 + 5t$. We initialize the particles in the local minima around $m_2=(3,...,3)$, drawing initial configurations from $\cN(m_2, \frac{1}{20}I)$. We run the Langevin versions for 500 steps of length $\lambda \Delta t = \frac{5}{2}\cdot \frac 1 {500}$. We run the PDSA versions at a timescale of $\lambda = 50$, and we estimate the velocity 25 times for the controlled versions. In this experiment we compute similar, best of $k=5$, objective values, shown in \cref{fig:optim_rastrigin} and \cref{fig:optim_rastrigin_boxplots}. The results show that the controlled Langevin version performs best, and in \cref{fig:optim_rastrigin_langevin_vs50} we find that simulating this controlled version with $5$ particles performs similar to taking the best of 50 independent trials.  

\begin{figure}
    \hspace{0.3in}
        \includegraphics{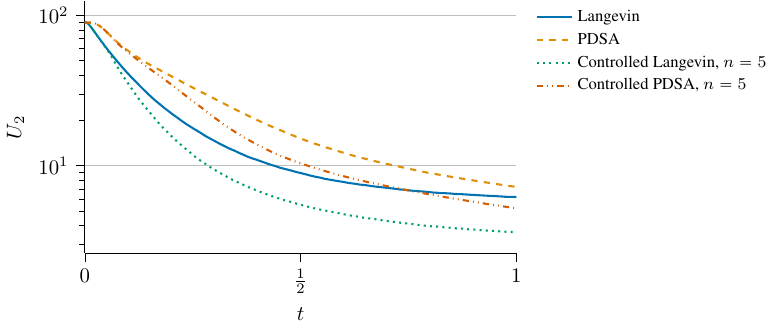}
        \caption{Results for the simulations on the Rastrigin potential $\obj_2$ in \eqref{eq:rastrigin}. Out of 2,000 total simulations with $k$ particles, we report the median best of $k=5$ objective value, $\min_{i\leq k} \obj_2(X^i_t)$, simulated independently or in groups of $5$. Here, the controlled Langevin version performs better than the other methods.} \label{fig:optim_rastrigin}
\end{figure}

\begin{figure}
    \hspace{0.3in}
    \includegraphics{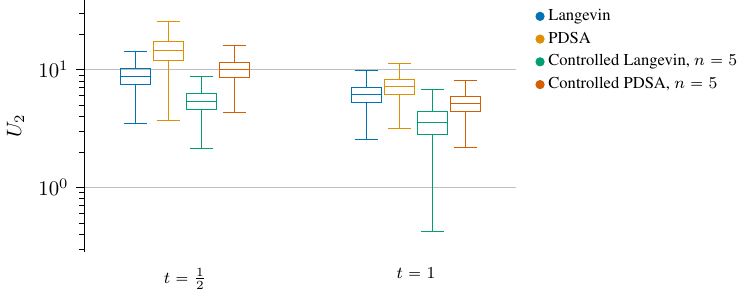}
    \caption{Box plots of the best objective values at $t=\frac{1}{2}$ and $1$, showing the run-to-run variance of the results from the Rastrigin simulations in \cref{fig:optim_rastrigin}. The controlled Langevin version performs best.} 
    \label{fig:optim_rastrigin_boxplots}
\end{figure}

\begin{figure}
    \hspace{0.3in}
        \includegraphics{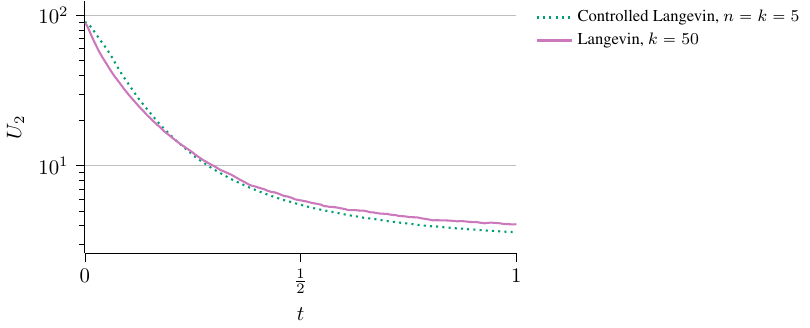}
        \caption{We once again compare the controlled Langevin version of \cref{fig:optim_rastrigin}, simulated in groups of $5$, to independent simulations. We plot the median best of $5$ for the controlled version, to the median best of $50$ independent simulations. In this case, superimposing 25 velocity estimations performs similar to a tenfold increase in independent starts.} \label{fig:optim_rastrigin_langevin_vs50}
\end{figure}

For the Rosenbrock problem, we found that the effect of estimating the flow was negligible compared to gradient descent dynamics.
We expect that this is due to the Rosenbrock problem being a locally difficult problem, in the sense that its difficulty lies in its poor scaling and conditioning. However, the most important effect of the (true) velocity field is the global flow between modes. In contrast to the Rosenbrock problem, the Rastrigin potential is highly multimodal. In this case, we find that superimposing the velocity estimations yields a strong increase in performance for the Langevin version. Specifically, we find that solving a few transport problems along the way gave an improvement comparable to a tenfold increase in restarts, illustrating that there is some increased communication between modes as a result of the introduced interactions.

\clearpage %

\section*{Acknowledgements}
This work was partially supported by the Wallenberg AI, Autonomous Systems and Software Program (WASP) funded by the Knut and Alice Wallenberg Foundation and partially supported by the Chalmers AI Research (CHAIR) project ``Stochastic Continuous-Depth Neural Networks''.

\bibliographystyle{plain}
\bibliography{refs.bib}

\appendix

\section{Deferred proofs}\label{app:proofs}

\subsection{Proofs of Section \ref{sec:v_exists_general_case}}

We verify that the Poincar\'e inequality used in the proof of \cref{thm:gibbs_ac} is bounded on $[0,T]$. We can convince ourselves that this is true via a Theorem in \cite{bakry2008simple}, but we repeat their construction of the Lyapunov function used for completeness.

\poincarelemma*

\begin{proof}%
    By \cite{bakry2008simple} we have that, if there exists a Lyapunov function $W\geq 1$ such that, for all $x$,
    \[
        \cL_tW(x) \leq -\theta W(x) + b\mathbf{1}_{B_R}(x)
    \]
    then
    \[
        \textnormal{Var}_{\mu_t}(f) \leq \frac{1}{\theta}(1+b\kappa_t) \int|\nabla f|^2 \dif \mu_t,
    \]
    where $\kappa_t$ is the Poincar\'e constant of $\mu_t$ restricted to the ball $B_R := \{x\colon |x| < R\}$ and
    $\theta = \gamma(\beta(t)\alpha-\gamma-\frac{n-1}{R}) > 0$ for some $\gamma > 0$.
    For any $R$, the following is a finite bound on the Poincar\'e constant of $\mu_t$ restricted to $B_R$
    \[
        \kappa_t \leq \exp\left({\beta(t)\left(\sup_{x\in B_R} \obj(x)-\inf_{x\in B_R}\obj(x)\right)}\right)R/\pi.
    \]
    We couple the simple Lyapunov function $W(x) = |x|^2 + 1$ with \cref{ass:quadratic_gradient}.
    Then
    \begin{align*}
        \cL_t W(x) & = -2\langle\nabla\obj(x),x\rangle + 2n                                         \\
                   & = \left(-2\langle\nabla\obj(x),x\rangle + 2n\right)I_{|x| > R}                 \\
                   & \phantom{=} + \left(-2\langle\nabla\obj(x),x\rangle + 2n\right)I_{|x| \leq R}  \\
                   & \leq \left(-2\alpha|x|^2 + 2n\right)I_{|x| > R}                                \\
                   & \phantom{=} +  \left(-2\langle\nabla\obj(x),x\rangle + 2n\right)I_{|x| \leq R} \\
                   & = 2\left(-\alpha + \frac{\alpha + n}{|x|^2 + 1}\right)W(x)I_{|x| > R}          \\
                   & \phantom{=} + \left(-2\langle\nabla\obj(x),x\rangle + 2n\right)I_{|x| \leq R}  \\
                   & \leq 2\left(-\alpha + \frac{\alpha + n}{R^2 + 1}\right)W(x)I_{|x| > R}         \\
                   & \phantom{=} + \left(-2\langle\nabla\obj(x),x\rangle + 2n\right)I_{|x| \leq R}.
    \end{align*}
    Possibly increasing $\tilde{R} := \max(R, \sqrt{\frac n \alpha} + 1)$ and setting $\theta = 2\left(\alpha - \frac{n + \alpha}{\tilde{R}^2 + 1}\right) > 0$ we find
    \begin{align*}
        \cL_t W & \leq -\theta W(x)I_{|x| > R} + \left(-2\langle\nabla\obj(x),x\rangle + 2n\right)I_{|x| \leq R}                       \\
                & = -\theta W(x)(1-I_{|x| \leq R}) + \left(-2\langle\nabla\obj(x),x\rangle + 2n\right)I_{|x| \leq R}                   \\
                & = -\theta W(x) + \left(\theta W(x) -2\langle\nabla\obj(x),x\rangle + \frac{2n}{\beta(t)} \right)I_{|x|\leq \tilde R}
    \end{align*}
    Then
    \[
        b_t = \max_{|x|\leq R} \theta + \theta|x|^2 -2\langle\nabla\obj(x),x\rangle + \frac{2n}{\beta(t)} < \infty.
    \]
    Since $b_0 > b_t$ for all $t$ we can pick $b := b_0$.
    With $K = \sup_{x\in B_R} \obj(x) - \inf_{x\in B_R} \obj(x)$ we then have
    \[
        \textnormal{Var}_{\mu_t}(f) \leq \frac{1}{\theta}(1+be^{\beta(t)K}\tilde{R}/\pi) \int |\nabla f |^2 \dif \mu_t
    \]
    for all $t$.
\end{proof}

Additionally, $\obj$ permits moments of any order.

\begin{lemma}[Bounded variance of $\obj$]\label{lem:u_bounded_variance}
Assume that $\obj$ satisfies \cref{ass:inf_obj_exists,ass:quadratic_tail,ass:quadratic_gradient}. Then, for all $\beta > 0$, $\obj$ permits $\mu_\beta$-moments of any order $n \geq 0$ in the sense that
\[
    \int_\Rdim (\obj(x))^n\pi_\beta(x) \dx < \infty.
\]
Specifically, it has bounded variance
\[
    \var_{\mu_\beta}(U) =\int_\Rdim \left(\obj(x)-\int \obj \dif \mu_\beta\right)^2\pi_\beta(x) \dx < \infty.
\]
\end{lemma}

\begin{proof}
    The moment-generating function of the random variable $\obj = \obj(X), X\sim\mu_\beta$ is given by
    \begin{align*}
        M_\obj(s) &= \EE_\obj\left[e^{s\obj}\right] = \EE_X\left[e^{s\obj(X)}\right] = \int_\Rdim \frac{e^{(s-\beta)\obj(x)}}{Z_\beta} \dx.
    \end{align*}
    Since $\beta > 0$ there exists $\delta > 0$ such that $(s-\beta) < 0$ for $|s|<\delta$. By the proof of \cref{lem:gibbs_in_Pp}, $M_\obj(s)$ is then well-defined for all $s \in (-\delta,\delta)$ which guarantees the existence of all moments.
\end{proof}

\subsection{Proofs of Section \ref{sec:one_dim_case}}\label{app:proofs_1d}

Here, we establish that the Gibbs curve is absolutely continuous on $\RR$. We first recall the definition of $v_t$ in \eqref{eq:1d_v},
\[
    v_t(x) = - \frac{1}{\pi_t(x)} \int_{-\infty}^x \partial_t \pi_t(y) \dy = \frac{\beta'(t)}{\pi_t(x)} \int_{-\infty}^x \left( \obj (y) - \int \obj \dif\mu_t \right) \pi_t(y) \dy.
\]

\thmonedimac*

The following elementary estimate controls the tails of $v_t$.

\begin{lemma}\label{lem:1d_v_tail_control}
    Under \cref{ass:1d_gradient_condition} the following estimates hold
    \begin{align}
        x < -R &\implies 
            \int_{-\infty}^x \obj(z)\pi_t(z) \dif z \leq \frac{1}{\alpha}\pi_t(x)\obj(x), \\
        x > R &\implies
            \int_{x}^\infty \obj(z)\pi_t(z) \dif z \leq \frac{1}{\alpha}\pi_t(x)\obj(x).
    \end{align}
\end{lemma}

\begin{proof}
    Without loss of generality, assume that \cref{ass:1d_gradient_condition} holds with $\alpha = 1$. Then, since 
    \[
        \frac{d}{\dx} (\obj e^{-\beta\obj}) = \left( -\beta \obj' + \frac{\obj'}{\obj} \right)Ue^{-\beta U},
    \]
    we find that, for $z < -R$, $\left( -\beta \obj' + \frac{\obj'}{\obj} \right) \geq 1$, so
    \begin{align*}
        \int_{x}^\infty \obj(z)e^{-\beta \obj(z)} \dif z &\leq \int_{-\infty}^x \left(-\beta \obj'(z) + \frac{\obj'(z)}{\obj(z)}\right)\obj(z)e^{-\beta \obj(z)} \dif z \\
        &= \obj(x)e^{-\beta \obj(x)}.
    \end{align*}
    Similarly, for $z > R$ we have that $-\left( -\beta \obj' + \frac{\obj'}{\obj} \right) \geq 1$ and hence
    \begin{align*}
        \int_{x}^\infty \obj(z)e^{-\beta \obj(z)} \dif z &\leq -\int_{x}^\infty \left(-\beta \obj'(z) + \frac{\obj'(z)}{\obj(z)}\right)\obj(z)e^{-\beta \obj(z)} \dif z \\
        &= \obj(x)e^{-\beta \obj(x)},
    \end{align*}
    and multiplying each side with $Z^{-1}_t$ finishes the argument.
\end{proof}

Equipped with the preceding lemma, we are now ready to prove the main result of \cref{sec:one_dim_case}.

\begin{proof}[Proof of \cref{thm:1d_ac}]
Writing $\bar U_t := \int \obj \dif \mu_t$ and $I_t(x) := \int_{-\infty}^x \left( \obj(z) - \bar\obj_t \right)\pi_t(z) \dif z$ so that $v_t(x) = \pi_t(x)^{-1}I_t(x)$. We note that \( I(x) \to_{|x|\to\infty} 0 \) and for sufficiently large $|x|$, $\obj(x)-\bar\obj_t > 0$. Hence
\begin{equation*}
    I'(x) = (\obj(x) - \bar U_t )\pi_t(x) > 0
\end{equation*}
so $I(x)$ is strictly increasing towards $0$ when $x\to\infty$ and decreasing when $x\to-\infty$.

For $x > R$ we find then
\begin{align*}
    |I(x)| = -I(x) &= \int_{-\infty}^x\bar\obj_t\pi_t(z)\dif z-\int_{-\infty}^x\obj(z)\pi_t(z)\dif z \\
    &= \bar\obj_t\int_{-\infty}^x\pi_t(z)\dif z - \left(\bar\obj_t - \int_x^\infty \obj(z) \pi_t(z)\dif z\right) \\
    &= \bar\obj_t\left(\int_{-\infty}^x \pi_t(z)\dif z - 1\right) + \int_x^\infty \obj(z)\pi_t(z)\dif z \\
    &= -\bar\obj_t\int_{x}^\infty \pi_t(z)\dif z + \int_x^\infty \obj(z)\pi_t(z)\dif z \\
    &\leq \int_x^\infty \obj(z) \pi_t(z)\dif z\leq \obj(x)\pi_t(x),
\end{align*}
where we applied Lemma \ref{lem:1d_v_tail_control} in the last step. Similarly, for $x < -R$ we find that
\begin{align*}
    |I(x)| = I(x) &= \int_{-\infty}^x (\obj(z)-\bar\obj_t)\pi_t(z)\dif z \\
    &\leq \int_{-\infty}^x \obj(z)\pi_t(z)\dif z \leq \obj(x)\pi_t(x).
\end{align*}
Then $|x| > R \implies |v_t(x)| \leq |\beta'(t)|U(x)$ and
\begin{align*}
    \|v_t\|^p_{L^p(\mu_t,\RR)} &= \int_\RR |v_t(x)|^p \pi_t(x) \dx \\
    &= |\beta'(t)|^p \int_{-R}^R (\pi_t(x))^{-p}|I(x)|^p \pi_t(x) \dx + \int_{|x|>R} |v_t(x)|^p\pi_t(x) \dx \\
    &\leq |\beta'(t)|^p \underbrace{\left( \int_{-R}^R (\pi_t(x))^{-p}|I(x)|^p \pi_t(x) \dx + \int_{|x|>R}(\obj(x))^p\pi_t(x) \dx\right)}_{:=M_t} \\
    &= |\beta'(t)|^pM_t < \infty.
\end{align*}

As $t\mapsto M_t$ is continuous, $\sup_{0<t<T} M_t < \infty$ and the curve is absolutely continuous.
\end{proof}

\end{document}